\pdfoutput=1
\documentclass[a4paper, 11pt, reqno]{amsart}

\usepackage{amssymb}
\usepackage{amsmath}
\usepackage{amsthm}
\usepackage[numbers]{natbib}
\usepackage{mathtools}
\usepackage{enumitem}
\usepackage{calc}
\usepackage{setspace}
\usepackage{graphicx}
\usepackage{tikz-cd}
\usepackage{hyperref}

\newcounter{dummy}
\numberwithin{dummy}{section}
\newtheorem{thm}[dummy]{Theorem}
\newtheorem{defn}[dummy]{Definition}
\newtheorem{lem}[dummy]{Lemma}
\newtheorem{prop}[dummy]{Proposition}

\newtheorem{conj}[dummy]{Conjecture}
\numberwithin{equation}{section}

\renewcommand\Re{\operatorname{Re}}
\renewcommand\Im{\operatorname{Im}}

\DeclareMathOperator{\id}{id}

\DeclareMathOperator{\Gal}{Gal}
\DeclareMathOperator{\GL}{GL}
\DeclareMathOperator{\SL}{SL}

\DeclareMathOperator{\Stab}{Stab}
\DeclareMathOperator{\End}{End}
\DeclareMathOperator{\M}{M}
\DeclareMathOperator{\Hom}{Hom}
\DeclareMathOperator{\pr}{pr}
\DeclareMathOperator{\Tr}{Tr}
\DeclareMathOperator{\tors}{tors}
\DeclareMathOperator{\an}{an}

\makeatletter
\renewcommand{\@biblabel}[1]{[#1]\hfill}
\makeatother

\begin{document}
\title{Unlikely intersections with isogeny orbits in a product of elliptic schemes}

\author{Gabriel A. Dill}
\date{\today}
\address{Departement Mathematik und Informatik, Universit\"{a}t Basel, Spiegelgasse~1, CH-4051 Basel}
\email{gabriel.dill@unibas.ch}

\begin{abstract}
Fix an elliptic curve $E_0$ without CM and a non-isotrivial elliptic scheme over a smooth irreducible curve, both defined over the algebraic numbers. Consider the union of all images of a fixed finite-rank subgroup (of arbitrary rank) of $E_0^g$, also defined over the algebraic numbers, under all isogenies between $E_0^g$ and some fiber of the $g$-th fibered power $\mathcal{A}$ of the elliptic scheme, where $g$ is a fixed natural number. As a special case of a slightly more general result, we characterize the subvarieties (of arbitrary dimension) inside $\mathcal{A}$ that have potentially Zariski dense intersection with this set. In the proof, we combine a generalized Vojta-R\'emond inequality with the Pila-Zannier strategy.
\end{abstract}


\maketitle
\tableofcontents
\section{Introduction}
Let $K$ be a field of characteristic zero and let $S$ be a smooth and geometrically irreducible curve, defined over $K$. Let $\pi: \mathcal{A} \to S$ be an abelian scheme of relative dimension $g$ over $S$, also defined over $K$. The zero section $S \to \mathcal{A}$ is called $\epsilon$. The morphism $\pi: \mathcal{A} \to S$ is by definition smooth and proper.

Let $\bar{K}$ be a fixed algebraic closure of $K$. All varieties that we consider will be defined over $\bar{K}$, if not explicitly stated otherwise. We will identify all varieties with their set of closed points over a prescribed algebraic closure of their field of definition. Subvarieties will always be closed. By ``irreducible'', we will always mean ``geometrically irreducible''. For a field extension $F$ of the field over which the variety $V$ is defined, the set of points of $V$ that are defined over $F$ is denoted by $V(F)$. The set of torsion points of an abelian variety $A$ is denoted by $A_{\tors}$ and its zero element by $0_A$ or just $0$ if there is no potential confusion.

If $s$ is any (possibly non-closed) point of $S$, we use a subscript $s$ to denote fibers over $s$. We denote the generic point of $S$ by $\xi$ and fix an algebraic closure $\overline{K(S)}$ of $\bar{K}(S)$. As mentioned above, we identify $\mathcal{A}_\xi$ with its closed points over $\overline{K(S)}$ and thus implicitly with its base change to $\overline{K(S)}$. Let $\left(\mathcal{A}_\xi^{\overline{K(S)}/\bar{K}},\Tr\right)$ denote the $\overline{K(S)}/\bar{K}$-trace of $\mathcal{A}_\xi$, as defined in Chapter VIII, \S 3 of \cite{MR0106225}. The abelian scheme $\mathcal{A}$ is called isotrivial if $\Tr$ is surjective.

Let $A_0$ be a fixed abelian variety of dimension $g$. We fix a finite set of $\mathbb{Z}$-linearly independent points $\gamma_1, \dots, \gamma_r$ in $A_{0}(\bar{K})$. The set can also be empty (i.e. $r=0$). We set
\[ \Gamma = \{ \gamma \in A_0; \mbox{ } \exists N \in \mathbb{N} \mbox{: } N\gamma \in \mathbb{Z}\gamma_1+\cdots+\mathbb{Z}\gamma_r \}, \]
a subgroup of $A_0$ of finite rank (and every subgroup of $A_0$ of finite rank is contained in a group of this form), for us $\mathbb{N} = \{1,2,3,\hdots\}$.

We define the isogeny orbit of $\Gamma$ (in the family $\mathcal{A}$) as
\begin{multline}
\mathcal{A}_{\Gamma} = \{ p \in \mathcal{A}_s; \mbox{ } s \in S, \mbox{ } \exists \phi: A_0 \to \mathcal{A}_s \mbox{ isogeny such that } p \in \phi(\Gamma) \}.
\end{multline}
This condition is equivalent to the existence of an isogeny $\psi: \mathcal{A}_s \to A_0$ with $\psi(p) \in \Gamma$.

The following is a special case of the main result of this article:

\begin{thm}\label{thm:mainmain}
Suppose that $K$ is a number field, that $\mathcal{A} \to S$ is not isotrivial, and that over $\overline{K(S)}$, $\mathcal{A}_\xi$ is isogenous to a power of an elliptic curve. Suppose further that $A_0$ is isogenous to $E_0^{g}$, where $E_0$ is an elliptic curve with $\End(E_0) = \mathbb{Z}$.

Let $\mathcal{V} \subset \mathcal{A}$ be an irreducible subvariety. If $\mathcal{A}_{\Gamma} \cap \mathcal{V}$ is Zariski dense in $\mathcal{V}$, then one of the following two conditions is satisfied:
\begin{enumerate}[label=(\roman*)]
\item The variety $\mathcal{V}$ is a translate of an abelian subvariety of $\mathcal{A}_s$ by a point of $\mathcal{A}_{\Gamma} \capÊ\mathcal{A}_s$ for some $s \in S$.
\item Over $\overline{K(S)}$, the variety $\mathcal{V}_\xi$ is a union of translates of abelian subvarieties of $\mathcal{A}_\xi$ by points in $(\mathcal{A}_\xi)_{\tors}$.
\end{enumerate}
\end{thm}

Compared to similar earlier results, a new aspect is that at once $\mathcal{V}$ is allowed to be of arbitrary dimension and $\Gamma$ of arbitrary rank. So far, results have been obtained only in the cases when $\mathcal{V}$ is a curve (Dill \cite{D18}, Gao \cite{G15}, Lin-Wang \cite{MR3383643}) or $\Gamma$ contains only torsion points (Gao \cite{G15}, Habegger \cite{MR3181568}, Pila \cite{MR3164515}). See also \cite{B18} and \cite{PT19} for related results.

If one tries to apply the arguments found in the literature to prove Theorem \ref{thm:mainmain}, the main stumbling block one encounters consists of obtaining a bound for the height of a point in $\mathcal{A}_\Gamma \cap \mathcal{V}$ (outside some degenerate locus) that depends polynomially on the degree of the point. This amounts to solving a Mordell-Lang problem in every fiber, but in a uniform way. Since the known height bounds for the Mordell-Lang problem are ineffective, this is a serious obstacle.

We solve this problem in Theorem \ref{thm:vojtaheightbound}, applying a generalized Vojta-R\'emond inequality in the form of Theorem \ref{thm:vojta} (Appendix \ref{sec:genvojtaineq}). The generalized Vojta-R\'emond inequality allows one to compare points from different isogenous fibers. The height bound is still ineffective, but the ineffectivity is now uniformly spread out over all fibers instead of occurring in each fiber separately. Once the height bound is obtained, we proceed along well-known tracks and apply the Pila-Zannier strategy, which is described in Zannier's book \cite{MR2918151} together with many problems that can be grouped under the umbrella of ``unlikely intersections".

Having an upper bound for the height that depends on the degree of the point is rather unusual compared to previous applications of Vojta's inequality and its generalizations. Theorem 1.3 in \cite{MR3578914} is another instance of such a bound that is even logarithmic in the degree of the point. Further examples can be found in \cite{MR2642162} and \cite{DO19}. In our situation, we only obtain a polynomial bound, but this is sufficient for the Pila-Zannier strategy.

If $\mathcal{A}$ is a constant abelian scheme over an irreducible projective base variety $S$ of arbitrary dimension, both defined over $\bar{\mathbb{Q}}$, then von Buhren has obtained in \cite{MR3614529} a similar height bound as the one we prove, bounding the height of a point $p$ (outside some degenerate locus) in terms of the height of $\pi(p)$. However, the fact that our result deals with varying abelian varieties rules out a direct application of \cite{MR3614529} or of R\'emond's generalized Vojta inequality in \cite{MR2233693}. The naive idea to just consider the image (or pre-image) under an isogeny of $\mathcal{V}_s$ in $A_0$ for varying $s \in S$ is ruled out since the degree of the resulting subvariety will in general grow as the degree of the isogeny grows, while the method needs a uniform bound on the degree of the subvariety to produce the desired height bound. Therefore, a generalized Vojta-R\'emond inequality is required to handle the family case.

The conditions we put on $A_0$ and $\mathcal{A}$ are necessary to obtain the height bounds in Section \ref{sec:heightbounds} insofar as they are crucial to obtain a lower bound for a certain intersection number in Lemma \ref{lem:voidbringer}. If we assume that $A_0$ and $\mathcal{A}$ are principally polarized, then two conditions are necessary for an argument like ours to work (see Section 5.5 of the forthcoming dissertation \cite{DDiss}): First, every cycle on $A_0$ has to be numerically equivalent to a $\mathbb{Q}$-linear combination of intersections of divisors. Second, for a fiber $\mathcal{A}_s$ ($s \in S$) that is isogenous to $A_0$ we have to be able to choose a polarized isogeny $\phi: \mathcal{A}_s \to A_0$ such that the index of $\phi^{-1}\End^s(A_0)\phi \cap \End^s(\mathcal{A}_s)$ in $\phi^{-1}\End^s(A_0)\phi$ is bounded independently of $sÊ\in S$. Here, $\End^s(A)$ denotes the additive group of endomorphisms of a principally polarized abelian variety $A$ that are fixed by the Rosati involution.

In the setting of the more general version of Theorem \ref{thm:mainmain} that we will prove, this second condition will actually only be satisfied on each isotypic factor of $\mathcal{A}_s$ (together with the corresponding isotypic factor of $A_0$) and we need further restrictions to make sure that an effective cycle on $\mathcal{A}_s$ decomposes into a sum of cartesian products of effective cycles on the isotypic factors (up to numerical equivalence).

The required lower bound for the intersection number can also be obtained under other technical restrictions on $\mathcal{A}$ and $A_0$ (see Section 5.5 of \cite{DDiss}), but the case of a fibered power of an elliptic scheme and a corresponding power of a fixed elliptic curve without CM seems to be the most natural one to treat. It is not clear to us if and how such a bound could be obtained in full generality. We emphasize that all parts of the proof apart from the bound in Lemma \ref{lem:voidbringer} can be applied with some necessary modifications to an arbitrary abelian scheme $\mathcal{A}$ over a base curve $S$.

Theorem \ref{thm:mainmain} is an instance of the following conjecture, which was formulated in \cite{D18} as Conjecture 1.1. It is a slightly modified version of Gao's Conjecture 1.2 in \cite{G15}, which he calls the Andr\'{e}-Pink-Zannier conjecture, in the case of a base curve. The field $K$ is now again arbitrary of characteristic $0$ and we place no restrictions on $\mathcal{A} \to S$ or $A_0$.

\begin{conj} (Modified Andr\'{e}-Pink-Zannier over a curve)\label{conj:APZ}
Suppose that $\mathcal{A} \to S$ is not isotrivial. Let $\mathcal{V}Ê\subset \mathcal{A}$ be an irreducible subvariety. If $\mathcal{A}_{\Gamma} \cap \mathcal{V}$ is Zariski dense in $\mathcal{V}$, then one of the following two conditions is satisfied:
\begin{enumerate}[label=(\roman*)]
\item The variety $\mathcal{V}$ is a translate of an abelian subvariety of $\mathcal{A}_s$ by a point of $\mathcal{A}_{\Gamma} \capÊ\mathcal{A}_s$ for some $s \in S$.
\item We have $\pi(\mathcal{V}) = S$ and over $\overline{K(S)}$, every irreducible component of $\mathcal{V}_\xi$ is a translate of an abelian subvariety of $\mathcal{A}_\xi$ by a point in $(\mathcal{A}_\xi)_{\tors} + \Tr\left(\mathcal{A}_\xi^{\overline{K(S)}/\bar{K}}(\bar{K})\right)$.
\end{enumerate}
\end{conj}

Conjecture \ref{conj:APZ} is also related to a conjecture of Zannier's (see \cite{G15}, Conjecture 1.4) and follows from Pink's Conjecture 1.6 in \cite{MR2166087} (see \cite{G15}, Section 8). We refer to \cite{D18}, Section 1, for a more detailed discussion and a comparison of this conjecture with the Andr\'e-Pink-Zannier conjecture; the part of the conclusion of the Andr\'e-Pink-Zannier conjecture that seems to be missing here has been proven in this case by Orr in \cite{MR3377393}. Conjecture \ref{conj:APZ} can be regarded as one relative version of the Mordell-Lang conjecture, proven for abelian varieties by Vojta \cite{MR1109352}, Faltings \cite{MR1307396}, and Hindry \cite{MR969244}, and in its most general form by McQuillan in \cite{MR1323985}, analogously to the relative Manin-Mumford results proven by Masser and Zannier in e.g. \cite{MR2766181}. 

Obstacles to proving a reasonable analogue of the conjecture for a base variety $S$ of dimension bigger than $1$ are on the one hand the already mentioned inequality between intersection numbers; on the other hand the obstacle that prevented Orr from establishing Theorem 1.2 in \cite{MR3377393} beyond the curve case (described on p. 213 of \cite{MR3377393}) rears its head as well.

From now on and throughout the rest of this article, we suppose that $K$ is a number field and take as $\bar{K} = \bar{\mathbb{Q}}$ its algebraic closure in $\mathbb{C}$. We can now state a slightly more general version of Theorem \ref{thm:mainmain}:

\begin{thm}\label{thm:main}
Suppose that $\mathcal{A} \to S$ is not isotrivial and that over $\overline{\bar{\mathbb{Q}}(S)}$, $\mathcal{A}_\xi$ is isogenous to a product of elliptic curves. Suppose further that $A_0$ is isogenous to $E_0^{g-g'}Ê\times E_1Ê\times \cdots \times E_{g'}$, where $0 \leq g' < g$, the $E_i$ are elliptic curves ($i=0,\hdots,g'$), and $\Hom(E_i,E_j) = \{0\}$ ($i \neq j$) as well as either $g-g' = 1$ or $\End(E_0) = \mathbb{Z}$. We also suppose that $\Hom\left(\mathcal{A}_\xi^{\overline{\bar{\mathbb{Q}}(S)}/\bar{\mathbb{Q}}},E_0\right) = \{0\}$.

If $\mathcal{A}_{\Gamma} \cap \mathcal{V}$ is Zariski dense in $\mathcal{V}$, then one of the following two conditions is satisfied:
\begin{enumerate}[label=(\roman*)]
\item The variety $\mathcal{V}$ is a translate of an abelian subvariety of $\mathcal{A}_s$ by a point of $\mathcal{A}_{\Gamma} \capÊ\mathcal{A}_s$ for some $s \in S$.
\item Over $\overline{\bar{\mathbb{Q}}(S)}$, the variety $\mathcal{V}_\xi$ is a union of translates of abelian subvarieties of $\mathcal{A}_\xi$ by points in $(\mathcal{A}_\xi)_{\tors} + \Tr\left(\mathcal{A}_\xi^{\overline{\bar{\mathbb{Q}}(S)}/\bar{\mathbb{Q}}}(\bar{\mathbb{Q}})\right)$.
\end{enumerate}
\end{thm}

The plan of this article is as follows: In Section \ref{sec:preliminaries}, we fix some notation. In Section \ref{sec:structuretheorem}, we show that it suffices to prove Conjecture \ref{conj:APZ} for $\mathcal{V}$ of a certain non-degenerate type without placing any restrictions on $\mathcal{A}$. In Section \ref{sec:heightbounds}, we apply a generalized Vojta-R\'emond inequality (see Appendix \ref{sec:genvojtaineq} and \cite{D182}) to deduce a height bound of the necessary form for a sufficiently large subset of $\mathcal{A}_{\Gamma} \cap \mathcal{V}$ if $\mathcal{V}$ is not degenerate and $\mathcal{A}$ and $A_0$ are of the form described in Theorem \ref{thm:main}. In Section \ref{sec:pila-zannier}, we apply the Pila-Zannier strategy and use the height bound we obtained in Section \ref{sec:heightbounds}. The necessary Ax-Lindemann-Weierstrass statement has been proven by Pila in \cite{MR3164515}. In Section \ref{sec:mainproof}, we put all the pieces together and prove Theorem \ref{thm:main}.

\section{Preliminaries and notation}\label{sec:preliminaries}
We fix a square root of $-1$ inside $\mathbb{C}$ that is denoted by $\sqrt{-1}$; this yields maps $\Re: \mathbb{C} \to \mathbb{R}$ and $\Im: \mathbb{C} \toÊ\mathbb{R}$ in the usual way. The upper half plane $\mathbb{H}$ is the set of $\tau \in \mathbb{C}$ with $\Im \tau > 0$. For an integral domain $R$, we denote the space of $m\times n$-matrices with entries in $R$ by $\M_{m \times n}(R)$. We write $\M_n(R)$ for $\M_{n \times n}(R)$. The complex conjugate of a matrix $A$ with complex entries will be denoted by $\overline{A}$. The $n$-dimensional identity matrix will be denoted by $I_n$. The maximum of the entries of a matrix $A \in \M_{m\times n}(\mathbb{C})$ in absolute value will be denoted by $\lVert A \rVert$. Vectors will always be column vectors.

We use the logarithmic absolute Weil height $h$ on projective space as defined in Definition 1.5.4 in \cite{MR2216774} by use of the maximum norm at the infinite places. The height of a finite subset of $\bar{\mathbb{Q}}$ is defined by considering it as a point in an appropriate projective space. If $V$ is a projective variety, possibly reducible, and $L$ a very ample line bundle on $V$, then any closed embedding of $V$ into projective space associated to $L$ yields an associated height $h_{V,L}$. It will always be clear from the context which embedding we choose.

Let $n_1, \hdots, n_q \in \mathbb{N}$ and let $V \subset \mathbb{P}^{n_1} \times \cdots \times \mathbb{P}^{n_q}$ be a subvariety. For $\alpha = (\alpha_1,\hdots,\alpha_q) \in (\mathbb{N} \cupÊ\{0\})^q$ such that $\alpha_1+\cdots+\alpha_q = \dim V + 1$, R\'emond defines a height $h_\alpha(V)$ in \cite{MR1837829}. In particular, if $q=1$ and $V \subset \mathbb{P}^{n}$ (with $n = n_1$), there is a height $h(V) = h_{\dim V +1}(V)$, coinciding with the one defined in \cite{MR1341770}. Similarly, if $\alpha_1 +Ê\cdots +\alpha_q = \dim V$, we use the degree $d_\alpha(V)$ as defined in \cite{MR1837829}.

If $V_1,\hdots,V_m$ are the irreducible components of $V$ of maximal dimension, then the height $h_\alpha(V)$ is the sum of the $h_\alpha(V_i)$: This follows from the definition of the resultant form in \cite{MR1837827}, p. 74. To apply the definition, a number field over which $V$ is defined has to be fixed, but the height is independent from the choice of the number field by Proposition 1.28 in \cite{MR3098424}. Furthermore, the height is always non-negative: This can be seen by applying Th\'eor\`eme 4 from \cite{MR1286891} several times to bound the contributions at the infinite places from below by the corresponding contributions in the height $\bf{h}$ as defined in \cite{MR876159} and then using Proposition 1.12(v) from \cite{MR876159}.

\section{Reduction to the non-degenerate case}\label{sec:structuretheorem}
In this section, we place no restrictions on $\mathcal{A} \toÊS$ except that $S$ should be a smooth irreducible curve. We keep our standing assumption that $K$ is a number field, although the results and proofs in this section are valid for arbitrary $K$ of characteristic $0$. We will show that it suffices to prove Conjecture \ref{conj:APZ} for a certain special type of $\mathcal{V}$ that one might call ``non-degenerate". In the proof of Proposition \ref{prop:reductionstep}, where this reduction is achieved, we will need to apply the conjecture to another abelian scheme and another abelian variety than the ones we started with. However, the new abelian scheme and abelian variety are obtained by a finite set of operations which preserve many properties of the abelian scheme. We formalize this process in the following definition.

\begin{defn}\label{defn:stable}
A non-empty set $\mathfrak{S}$ of isomorphism classes of pairs ($\mathcal{A} \to S$, $A_0$) of abelian schemes $\mathcal{A}Ê\to S$ with $S$ smooth and irreducible and $\dim S = 1$ and abelian varieties $A_0$ that are isogenous to infinitely many fibers of $\mathcal{A}$ is called stable if it has the following properties:
\begin{enumerate}[label=(\roman*)]
\item If $S'$ is smooth and irreducible, $\dim S' = 1$, the isomorphism class of $(\mathcal{A}Ê\to S,A_0)$ is in $\mathfrak{S}$, and there is a quasi-finite morphism $S' \to S$, then the isomorphism class of $(\mathcal{A} \times_{S}ÊS' \to S',A_0)$ is in $\mathfrak{S}$.
\item If the isomorphism class of $(\mathcal{A} \to S,A_0)$ is in $\mathfrak{S}$, $\mathcal{A}' \to S$ is an abelian scheme whose generic fiber is isogenous (over $\bar{\mathbb{Q}}(S)$) to the generic fiber of $\mathcal{A}Ê\to S$, and $A_0'$ is an abelian variety that is isogenous to $A_0$, then the isomorphism class of $(\mathcal{A}' \to S,A_0')$ is in $\mathfrak{S}$.
\end{enumerate}
\end{defn}

Here, two pairs ($\mathcal{A}Ê\to S$, $A_0$) and ($\mathcal{A}' \to S'$, $A'_0$) are called isomorphic if there exists an isomorphism $S \to S'$ of algebraic curves over $\bar{\mathbb{Q}}$, an isomorphism $\mathcal{A} \to \mathcal{A}' \times_{S'} S$ of abelian schemes over $S$, and an isomorphism of abelian varieties $A_0 \to A'_0$.

The following technical lemma shows that we can fix the isogeny in the definition of $\mathcal{A}_{\Gamma}$.

\begin{lem}\label{lem:technicallemma}
For each $s \in S$ such that $\mathcal{A}_{s}$ and $A_0$ are isogenous, fix an isogeny $\phi_s: A_0 \to \mathcal{A}_s$. Let $\Gamma$ be a subgroup of $A_0$ of finite rank that is mapped into itself by $\End(A_0)$ and equal to its division closure. Then we have
\[\mathcal{A}_{\Gamma} = \{ p \in \mathcal{A}_s; \mbox{ } s \in S, \mbox{ $\mathcal{A}_{s}$ and $A_0$ isogenous and } p \in \phi_s(\Gamma) \}. \]
\end{lem}

\begin{proof}
See \cite{D18}, Lemma 2.2 (with $k = \dim A_0$).
\end{proof}

\begin{prop}\label{prop:reductionstep}
Let $\mathfrak{S}$ be a stable set of isomorphism classes of pairs of abelian schemes and abelian varieties. Suppose that Conjecture \ref{conj:APZ} is true for all $(\mathcal{A}Ê\to S,A_0)$ whose class lies in $\mathfrak{S}$ under the additional condition that the union of all translates of positive-dimensional abelian subvarieties of $\mathcal{A}_s$ that are contained in $\mathcal{V}_s$ for some $sÊ\in S$ is not Zariski dense in $\mathcal{V}$. Then it also holds unconditionally for all $(\mathcal{A}Ê\to S,A_0)$ whose class lies in $\mathfrak{S}$.
\end{prop}

If $A$ is an abelian variety, defined over any field, and $B \subset A$ a subvariety, we denote its stabilizer by
\[Ê\Stab(B,A) := \{ a \in A; a + BÊ\subset B\}.\]
As $\Stab(B,A) = \bigcap_{b \in B}{(-b + B)}$, it is Zariski closed. By considering each irreducible component of $B$ separately, we find that $a \in \Stab(B,A)$ actually implies that $a + B = B$. Hence, it is also closed under addition and inversion, and therefore is an algebraic subgroup of $A$. We can now state the following lemma, which we will need to prove Proposition \ref{prop:reductionstep}. Recall that $\xi$ denotes the generic point of $S$.

\begin{lem}\label{lem:not-zariskidense}
Suppose that all abelian subvarieties of $\mathcal{A}_\xi$ are defined over $\bar{\mathbb{Q}}(S)$ and that the stabilizer $\Stab(\mathcal{V}_\xi,\mathcal{A}_\xi)$ is finite. Then the union of all translates of positive-dimensional abelian subvarieties of $\mathcal{A}_s$ that are contained in $\mathcal{V}_s$ for some $s \in S$ is not Zariski dense in $\mathcal{V}$.
\end{lem}

Note that this lemma can also be obtained as a consequence of the much more general Theorem 12.2 in \cite{G152}, at least for $\mathcal{A}$ contained in a suitable universal family and then for arbitrary $\mathcal{A}$ as well. However, we have thought it worthwhile to include a simple proof here that does not make use of the language of mixed Shimura varieties.

\begin{proof}
We first pass to a finite flat cover $S' \to S$ such that $S'$ is smooth and irreducible and every (geometrically) irreducible component of $\mathcal{V}_\xi$ is defined over $\bar{\mathbb{Q}}(S')$. Set $\mathcal{A}' = \mathcal{A} \times_{S} S'$. Let $\mathcal{V}'$ be an irreducible component of $\mathcal{V}Ê\times_{S} S' \hookrightarrow \mathcal{A}'$. Since the morphism $\mathcal{V}Ê\times_{S} S' \to \mathcal{V}$ is flat as the base change of the flat morphism $S' \to S$, we know by Proposition 2.3.4(iii) in \cite{MR0199181} that $\mathcal{V}'$ dominates $\mathcal{V}$ and therefore must dominate $S'$.

If $\eta$ is the generic point of $S'$ and we identify $\mathcal{A}'_{\eta}$ with $\mathcal{A}_\xi$ (both being identified with their base change to $\overline{\bar{\mathbb{Q}}(S)}$), then $\Stab(\mathcal{V}'_{\eta},\mathcal{A}'_{\eta})$ must be finite. Otherwise it would contain a positive-dimensional abelian subvariety of $\mathcal{A}_\xi$, but as all abelian subvarieties of $\mathcal{A}_\xi$ are defined over $\bar{\mathbb{Q}}(S)$, this abelian subvariety would be contained in the stabilizer of $\mathcal{V}_\xi$, which could therefore not be finite. Furthermore, $\mathcal{V}'_\eta = \mathcal{V}' \cap \mathcal{A}'_\eta$ is irreducible by Section 2.1.8 of Chapter 0 of \cite{MR0163908} and hence geometrically irreducible by our choice of $S'$.

If the union of all translates of positive-dimensional abelian subvarieties of $\mathcal{A}_s$ that are contained in $\mathcal{V}_s$ for some $s \in S$ is Zariski dense in $\mathcal{V}$, then the union of all translates of positive-dimensional abelian subvarieties of $\mathcal{A}'_s$ that are contained in $\mathcal{V}'_s$ for some $s \in S'$ is Zariski dense in $\mathcal{V}'$. So we can replace $\mathcal{A}$ and $\mathcal{V}$ by $\mathcal{A}'$ and $\mathcal{V}'$ and assume without loss of generality that $\mathcal{V}_\xi$ is geometrically irreducible.

Let $N \in \mathbb{N}$ be a natural number that is larger than the order of $\Stab(\mathcal{V}_\xi,\mathcal{A}_\xi)$. There are finitely many closed irreducible curves $\mathcal{T}_1, \hdots, \mathcal{T}_R \subset \mathcal{A}$ such that the union of the $\mathcal{T}_i$ ($i=1,\hdots,R$) is equal to the set of points of exact order $N$ on $\mathcal{A}$: First of all, every irreducible component of the pre-image of $\epsilon(S)$ under the multiplication-by-$N$ morphism $[N]$ dominates $S$ by Proposition 2.3.4(iii) in \cite{MR0199181} since $[N]$ is \'etale, so flat (see \cite{MR861974}, Proposition 20.7). Therefore, every irreducible component of $[N]^{-1}\left(\epsilon(S)\right)$ is of dimension $1$. The same holds for any $M \in \mathbb{N}$ that divides $N$. Furthermore, $[N]^{-1}(\epsilon(S))$ is smooth as $[N]$ is \'etale and $S$ is smooth. Hence, no two distinct irreducible components of $[N]^{-1}\left(\epsilon(S)\right)$ intersect each other. So every irreducible component of $[N]^{-1}\left(\epsilon(S)\right)$ is either contained in $\bigcup_{M|N, M \neq N}[M]^{-1}\left(\epsilon(S)\right)$ or disjoint from it and our claim follows.

We now consider $\mathcal{W}_i = \mathcal{V} \cap (\mathcal{V}+\mathcal{T}_i)$ for $i \in \{1,\hdots,R\}$. If this variety were equal to $\mathcal{V}$, then we would have $\mathcal{V} \subset \mathcal{V}+\mathcal{T}_i$ and so $\mathcal{V}_\xi \subset \mathcal{V}_\xi+(\mathcal{T}_i)_\xi$. For dimension reasons and thanks to the (geometric) irreducibility of $\mathcal{V}_\xi$, we would get that $\mathcal{V}_\xi = t+\mathcal{V}_\xi$ for a torsion point $t \in \mathcal{A}_\xi$ of order $N$. This contradicts our choice of $N$. So $\mathcal{W}_i \subsetneq \mathcal{V}$.

On the other hand, each positive-dimensional abelian variety contains a point of order $N$, so the union of all translates of positive-dimensional abelian subvarieties of $\mathcal{A}_s$ that are contained in $\mathcal{V}_s$ for some $s \in S$ is contained in $\bigcup_{i=1}^{R}{\mathcal{W}_i}$. As every $\mathcal{W}_i$ is a proper closed subset of $\mathcal{V}$ and $\mathcal{V}$ is irreducible, the lemma follows.
\end{proof}

\begin{proof} (of Proposition \ref{prop:reductionstep})
We can assume without loss of generality that $\pi(\mathcal{V}) = S$ (else the conjecture reduces to the Mordell-Lang conjecture, proven by Faltings, Vojta, and Hindry). After a finite flat base change $S' \to S$ with $S'$ smooth and irreducible and after replacing $\mathcal{A}$ by $\mathcal{A} \times_{S}ÊS'$ and $\mathcal{V}$ by an irreducible component of $\mathcal{V} \times_{S} S'$, we can assume that all abelian subvarieties of $\mathcal{A}_\xi$ are defined over $\bar{\mathbb{Q}}(S)$.

Let $A'$ be the irreducible component of $\Stab(\mathcal{V}_\xi,\mathcal{A}_\xi)$ that contains $0_{\mathcal{A}_\xi}$. Then $A'$ is an abelian subvariety of $\mathcal{A}_\xi$. We can now use the Poincar\'{e} reducibility theorem to deduce that there exists another abelian subvariety $A''$ of $\mathcal{A}_\xi$ such that the natural morphism $A'Ê\times A'' \to \mathcal{A}_\xi$ given by restricting the addition morphism $\mathcal{A}_\xi \times \mathcal{A}_\xi \to \mathcal{A}_\xi$ is an isogeny.

The Zariski closures of $A'$ and $A''$ inside $\mathcal{A}$ are abelian schemes $\mathcal{A}'$ and $\mathcal{A}''$ over $S$ with $A' = \mathcal{A}'_\xi$ and $A'' = \mathcal{A}''_\xi$ by Corollary 6 in Section 7.1 and Proposition 2 in Section 1.4 of \cite{MR1045822}. The isogeny between the generic fibers extends to a morphism $\alpha: \mathcal{A}'Ê\times_{S} \mathcal{A}'' \to \mathcal{A}$, obtained by restricting the addition morphism. We denote the structural morphism $\mathcal{A}' \times_{S} \mathcal{A}'' \to S$ also by $\pi$.

As $\alpha$ is dominant, proper, and maps the image of the zero section to the image of the zero section, it follows that $\alpha$ restricts to an isogeny on each fiber. Let $\mathcal{V}'$ be an irreducible component of $\alpha^{-1}(\mathcal{V})$ that dominates $\mathcal{V}$. Under the hypotheses of Conjecture \ref{conj:APZ}, the intersection of $\mathcal{V}'$ with the set $(\mathcal{A}' \times_{S} \mathcal{A}'')_{\Gamma}$ is Zariski dense in $\mathcal{V}'$, so it suffices to prove the conjecture for $\mathcal{V}'$.

Let $\mathcal{V}''$ be the image of $\mathcal{V}'$ under the projection to $\mathcal{A}''$. Since the projection morphism is proper, $\mathcal{V}''$ is closed in $\mathcal{A}''$. By construction, the generic fiber of $\alpha^{-1}(\mathcal{V})$ contains $\mathcal{A}'_\xi \times \mathcal{V}''_\xi$. It follows that $\mathcal{V}' = \mathcal{A}' \times_{S} \mathcal{V}''$.

Since $A_0$ contains only finitely many abelian subvarieties up to automorphism (this is the main result of \cite{MR1378542}, due to Bertrand for algebraically closed fields of characteristic $0$), we can deduce that there exists an abelian subvariety $A_0' \subset A_0$ with the following property: The set of $p = \phi(\gamma) \in \mathcal{V}'$, where $\gamma \in \Gamma$ and $\phi: A_0 \to \mathcal{A}'_{\pi(p)} \times \mathcal{A}''_{\pi(p)}$ is an isogeny, such that there exists an automorphism of $A_0$ that maps $A_0'$ onto the irreducible component of $\phi^{-1}\left(\mathcal{A}'_{\pi(p)} \times \left\{0_{\mathcal{A}''_{\pi(p)}}\right\} \right)$ containing $0_{A_0}$ is Zariski dense in $\mathcal{V}'$. Using the Poincar\'{e} reducibility theorem over $\bar{\mathbb{Q}}$, we find an abelian subvariety $A_0''Ê\subset A_0$ such that the natural morphism $A_0' \times A_0'' \to A_0$ is an isogeny.

The pre-image of $\Gamma$ under this morphism is again a group of finite rank. It is contained in a group $\Gamma' \times \Gamma''$, where $\Gamma', \Gamma''$ are subgroups of finite rank of $A_0', A_0''$ respectively and $\Gamma' \times \Gamma''$ is stable under $\End(A_0' \times A_0'')$ and equal to its division closure.

Applying Lemma \ref{lem:technicallemma}, we find that the intersection of $\mathcal{V}'$ with the set
\begin{align*}
\{ (p,q) \in \mathcal{A}'_s \times \mathcal{A}''_s; \mbox{ } s \in S, \mbox{ } \exists \phi': A_0' \to \mathcal{A}'_s,  \phi'': A_0'' \to \mathcal{A}''_s \\
\mbox{isogenies such that } (p,q) \in \phi'(\Gamma') \times \phi''(\Gamma'') \}
\end{align*}
is Zariski dense in $\mathcal{V}'$. But this implies that the intersection of $\mathcal{V}''$ with the set
\[ \{ p \in \mathcal{A}''_s; \mbox{ } s \in S, \mbox{ } \exists \phi: A_0'' \to \mathcal{A}''_s \mbox{ isogeny such that } p \in \phi(\Gamma'') \} \]
is Zariski dense in $\mathcal{V}''$. Let $\epsilon': SÊ\to \mathcal{A}'$ be the zero section of $\mathcal{A}'$ and set $\mathcal{V}''' = \epsilon'(S) \times_{S} \mathcal{V}'' \subset \mathcal{A}' \times_{S}Ê\mathcal{A}''$.

By Lemma \ref{lem:not-zariskidense}, we now know by hypothesis that Conjecture \ref{conj:APZ} is true for $\mathcal{A}' \times_{S} \mathcal{A}'', \mathcal{V}''', A_0'Ê\times A_0'', \{0_{A_0'}\}Ê\times \Gamma''$ since $\Stab(\mathcal{V}''_\xi, \mathcal{A}''_\xi)$ and hence $\Stab(\mathcal{V}'''_{\xi},\mathcal{A}'_\xi \times \mathcal{A}''_\xi)$ must be finite by construction. We obtain that every irreducible component of $\mathcal{V}'''_\xi$ is a translate of an abelian subvariety of $\mathcal{A}'_\xi \times \mathcal{A}''_\xi$ by a point in $( \mathcal{A}'_\xiÊ\times \mathcal{A}''_\xi)_{\tors} + \Tr'\left(( \mathcal{A}'_\xi \times \mathcal{A}''_\xi)^{\overline{\bar{\mathbb{Q}}(S)}/\bar{\mathbb{Q}}}(\bar{\mathbb{Q}})\right)$, where $\left((\mathcal{A}'_\xi \times \mathcal{A}''_\xi)^{\overline{\bar{\mathbb{Q}}(S)}/\bar{\mathbb{Q}}}, \Tr'\right)$ denotes the $\overline{\bar{\mathbb{Q}}(S)}/\bar{\mathbb{Q}}$-trace of $\mathcal{A}'_\xi \times \mathcal{A}''_\xi$. As we have $\mathcal{V}''' = \epsilon'(S) \times_{S} \mathcal{V}''$ and $\mathcal{V}' = \mathcal{A}'Ê\times_{S} \mathcal{V}''$, we deduce the analogous statement for $\mathcal{V}'$.
\end{proof}

The following lemma will be useful to prove Conjecture \ref{conj:APZ} once we have established that the union of all weakly special curves that dominate $S$ and are contained in $\mathcal{V}$ lies Zariski dense in $\mathcal{V}$.

\begin{lem}\label{lem:weaklyspecialdense}
Suppose that there exists a subgroup $\Gamma' \subset \mathcal{A}_\xi^{\overline{\bar{\mathbb{Q}}(S)}/\bar{\mathbb{Q}}}(\bar{\mathbb{Q}})$ of finite rank such that the irreducible curves $\mathcal{C}$ that are contained in $\mathcal{V}$, dominate $S$, and satisfy $\mathcal{C}_\xi \subset (\mathcal{A}_\xi)_{\tors} + \Tr\left(\Gamma'\right)$ lie Zariski dense in $\mathcal{V}$. Then every irreducible component of $\mathcal{V}_\xi$ is a translate of an abelian subvariety of $\mathcal{A}_\xi$ by a point in $(\mathcal{A}_\xi)_{\tors} + \Tr\left(\mathcal{A}_\xi^{\overline{\bar{\mathbb{Q}}(S)}/\bar{\mathbb{Q}}}(\bar{\mathbb{Q}})\right)$.
\end{lem}

\begin{proof}
Apply the usual Mordell-Lang conjecture over the field $\overline{\bar{\mathbb{Q}}(S)}$ -- which is a consequence of Faltings' main theorem in \cite{MR1307396} together with Hindry's Section 6 and Proposition C in \cite{MR969244} -- to each irreducible component of $\mathcal{V}_\xi$.
\end{proof}

\section{Height bounds}\label{sec:heightbounds}
In this section, we assume that $A_0 = E_0^{g-g'}Ê\times E_1Ê\times \cdots \times E_{g'}$ for elliptic curves as in the hypothesis of Theorem \ref{thm:main}. Set $S = Y(2) = \mathbb{A}^1\backslash\{0,1\}$ and let 
\[ \mathcal{E} =Ê\{ (\lambda,[x:y:z]) \in Y(2) \times \mathbb{P}^2; y^2z = x(x-z)(x-\lambda z)\} \]
be the Legendre family of elliptic curves over $Y(2)$. We also assume that $\mathcal{A} = (\mathcal{E}Ê\times_S \cdots \times_{S} \mathcal{E}) \times_{S} (E_1 \times \cdots \times E_{g'} \times S)$ and $\pi(\mathcal{V}) = S$. We will show in Section \ref{sec:mainproof} that one can always assume this under the hypotheses of Theorem \ref{thm:main}.

There is a canonical open immersion of $S$ into $\mathbb{P}^1$. By choosing a Legendre model for $E_1$, \dots, $E_{g'}$, we obtain an immersion of $\mathcal{A}$ into $\mathbb{P}^1 \times \left(\mathbb{P}^2\right)^{g}$. Composing this with first the Veronese embedding of $\mathbb{P}^2$ into $\mathbb{P}^5$ and then the Segre embedding, we obtain an immersion of $\mathcal{A}$ into $\mathbb{P}^1 \times \mathbb{P}^{R}$ with $R = 6^g-1$. This induces very ample line bundles $L_{\overline{S}}$ on the compactification $\overline{S} = \mathbb{P}^1$ of $S$ and $\mathcal{L}$ on the associated compactification $\overline{\mathcal{A}}$ of $\mathcal{A}$ such that for each $s \in S$, the line bundle $\mathcal{L}$ restricts to the sixth power of an ample symmetric line bundle that induces a principal polarization on $\mathcal{A}_s$. (We use the Veronese embedding to obtain an even power of an ample line bundle -- this will be important in the proof of Lemma \ref{lem:isogenycontrol}.)

By choosing a Legendre model for each of its factors, we can embed $A_0$ into $\left(\mathbb{P}^2\right)^{g}$ and then as above in $\mathbb{P}^{R}$ and obtain a symmetric very ample line bundle $L_0$ on $A_0$, which is the sixth power of an ample symmetric line bundle that induces a principal polarization on $A_0$. By applying a linear automorphism on each factor $\mathbb{P}^5$ of $\left(\mathbb{P}^5\right)^{g}$, we can assume that all coordinate hyperplanes intersect $A_0$ transversally. When embedding $\mathcal{A}$ into $\mathbb{P}^1Ê\times \mathbb{P}^R$, we can choose the same embeddings into $\mathbb{P}^5$ for $E_1$, \dots, $E_{g'}$ as for the corresponding factors of $A_0$. The embeddings of $\mathcal{E}_s$, $E_0$, $E_1$, $\hdots$, $E_{g'}$ into $\mathbb{P}^5$ yield divisors on these curves. As the zero element is mapped to an inflection point of a plane curve in the Legendre model, these divisors are linearly equivalent to six times the zero element of the respective elliptic curve.

We get a (logarithmic projective) height $h_{A_0}=h_{A_0,L_0}$ on $A_0$. With the usual construction due to N\'{e}ron and Tate (see \cite{MR1745599}, Theorem B.5.1) we then obtain a canonical height $\widehat{h}_{A_0}$ on $A_0$. By our choice of embedding, we have $\widehat{h}_{A_0} = \sum_{i=1}^{g-g'}{\widehat{h}_{E_0} \circ \pi_i}+\sum_{i=1}^{g'}{\widehat{h}_{E_i} \circÊ\pi_{g-g'+i}}$, where $\widehat{h}_{E_j}$ denotes the canonical height on $E_j$ associated to its embedding in $\mathbb{P}^5$ ($j = 0,\hdots,g'$) and $\pi_k$ denotes the projection onto the $k$-th factor ($k=1,\hdots,g$).

After maybe enlarging the number field $K$, we can assume that $\mathcal{A}$ (with its structure as an abelian scheme), $\mathcal{L}$, $A_0$ (with its structure as an abelian variety), $L_0$ as well as $\mathcal{V}$ and the chosen immersions are all defined over $K$, $\gamma_1, \hdots, \gamma_r \in A_0(K)$, and every endomorphism of $A_0$ is defined over $K$. Since the endomorphism ring of $A_0$ is finitely generated as a $\mathbb{Z}$-module, we may assume that $\Gamma$ is mapped into itself by every endomorphism of $A_0$ by enlarging $\Gamma$ if necessary. We will generally assume that $r \geq 1$ for simplicity; one can either ensure this by enlarging $\Gamma$ and $K$ or one can check that our proof also works \emph{mutatis mutandis} if $r=0$.

The line bundle $L_{\overline{S}}$ on $\overline{S}$ yields a height $h_{\overline{S}}$ on $\overline{S}$. For each $s \in S$, the restriction of $\mathcal{L}$ to $\mathcal{A}_s$ is a very ample symmetric line bundle $\mathcal{L}_s$, which yields a height $h_s$, induced by the projective embedding $\mathcal{A}_s \hookrightarrow \mathbb{P}^R$, and a canonical height $\widehat{h}_s$ on $\mathcal{A}_s$. As for $\widehat{h}_{A_0}$, this height decomposes as $\widehat{h}_s = \sum_{i=1}^{g-g'}{\widehat{h}^{0}_{s} \circ \pi_i}+\sum_{i=1}^{g'}{\widehat{h}_{E_i} \circÊ\pi_{g-g'+i}}$, where $\widehat{h}^{0}_{s}$ denotes the canonical height on $\mathcal{E}_s$ associated to its embedding into $\mathbb{P}^5$ and -- by abuse of language -- $\pi_j$ again denotes the projection onto the $j$-th factor ($j=1,\hdots,g$). We will denote by $h^0_s$ and $h_{E_j}$ the usual, not necessarily canonical heights on $\mathcal{E}_s$ and $E_j$ ($j=0,\hdots,g'$) induced by the embeddings into $\mathbb{P}^5$.

By Lemma \ref{lem:technicallemma}, we can fix an isogeny $\phi_s$ in the definition of $\mathcal{A}_\Gamma$ for each $s \in S$ such that $\mathcal{A}_s$ and $A_0$ (or equivalently $\mathcal{E}_s$ and $E_0$) are isogenous. We take $\phi_s = (\psi_s,\hdots,\psi_s,d_s \cdot \id_{E_1\times\cdots\times E_{g'}})$, where $\psi_s$ is an isogeny from $E_0$ to $\mathcal{E}_{s}$ of minimal degree, i.e. there exists no isogeny of smaller degree between them, and $d_s = [\sqrt{\degÊ\psi_s}]$. Here and in the following, $[\alpha]$ denotes the largest integer less than or equal to $\alpha$ for any real number $\alpha$.

By Th\'{e}or\`{e}me 1.4 of Gaudron-R\'{e}mond in \cite{MR3225452}, which improves a theorem of Masser-W\"ustholz (\cite{MR1217345}, p. 460, and -- non-explicitly -- \cite{MR1037140}, p. 1), we have
\begin{equation}\label{eq:masserwuestholz}
\deg \psi_s \preceq [K(s):K],
\end{equation}
where we will write $f \preceq g$ for (positive) quantities $f$ and $g$ if there exist constants $c>0$ and $\kappa>0$, depending on $K$, $A_0$, $\Gamma$, $\mathcal{A}$, $S$, and $\mathcal{V}$ as well as $\mathcal{L}$, $L_{\overline{S}}$, $L_0$, and the immersions associated to these such that
\[ f \leq c\max\{1,g\}^{\kappa}.\]
Note that $\mathcal{A}_s$ and $A_0$ are both defined over $K(s)$.

All numbered constants $\tilde{c}_1, \tilde{c}_2, \hdots$ in the following will depend only on the quantities that the implicit constants in $\preceq$ are allowed to depend on.

\begin{lem}\label{lem:heightbound}
Let $s \in S$ be such that $\mathcal{A}_s$ and $A_0$ are isogenous. Then there exists a constant $\tilde{c}_1$ such that
\[ h_{\overline{S}}(s) \leq \tilde{c}_1\max\{\log[K(s):K],1\}.\]
\end{lem}

\begin{proof}
We denote the (stable) Faltings height of an abelian variety $A$ as defined in \cite{MR718935} by $h_F(A)$ and the $j$-invariant of an elliptic curve $E$ by $j(E)$.

By Faltings' Lemma 5 in \cite{MR718935}, we have
\begin{equation}\label{eq:faltingsestimate}
h_F(\mathcal{A}_s) \leq  h_F(A_0)+\frac{\log \deg \phi_s}{2}.
\end{equation}

The Faltings height of a product is the sum of the Faltings heights and we have a bound on the difference between $h_F(\mathcal{E}_s)$ and $\frac{1}{12}h(j(\mathcal{E}_s))$ due to Silverman (\cite{MR861979}, Proposition 2.1). Finally, the map $s \mapsto j(\mathcal{E}_s)$ extends to a non-constant morphism from $\overline{S} = \mathbb{P}^1$ to $\mathbb{P}^1$ and so we can bound $h_{\overline{S}}(s)$ from above by some multiple of $\max\{h(j(\mathcal{E}_s)),1\}$ using standard height estimates.

We deduce that
\begin{equation}\label{eq:j-faltings}
h_{\overline{S}}(s) \leq C\max\{h_{F}(\mathcal{A}_s),1\}
\end{equation}
for some constant $C$ that depends only on $E_1, \hdots, E_{g'}$. Combining \eqref{eq:masserwuestholz}, \eqref{eq:faltingsestimate}, and \eqref{eq:j-faltings}, we obtain that
\[ h_{\overline{S}}(s) \leq \tilde{c}_1\max\{\log[K(s):K],1\}\]
for some constant $\tilde{c}_1$.
\end{proof}

\begin{thm}\label{thm:vojtaheightbound}
Suppose that $\pi(\mathcal{V}) = S$. Let $s \in S$. Then $h_s(p) \preceq [K(s):K]$ for every point $p \in \mathcal{V}_s \cap \mathcal{A}_\Gamma$ that does not lie in a translate of a positive-dimensional abelian subvariety of $\mathcal{A}_{s}$ contained in $\mathcal{V}_s$.
\end{thm}

The proof of Theorem \ref{thm:vojtaheightbound} will occupy the rest of this section. Thanks to Lemma \ref{lem:heightbound} and \eqref{eq:masserwuestholz}, it suffices to find a bound that is polynomial in $\deg \psi_s$ (or equivalently $d_s$) and $h_{\overline{S}}(s)$. We can assume without loss of generality that $m := \dim \mathcal{V} \geq 2$, otherwise Theorem \ref{thm:vojtaheightbound} follows from Lemma \ref{lem:heightbound} and elementary height bounds due to the fact that $\pi|_{\mathcal{V}}: \mathcal{V} \to S$ is quasi-finite.

Let $s_1,\hdots,s_{m} \in S$, then $\mathcal{A}_{s_1}, \hdots, \mathcal{A}_{s_{m}}$ are abelian varieties with an embedding into $\mathbb{P}^R$ (by projection to the second factor of $\mathbb{P}^1 \times \mathbb{P}^R$). Assume that $A_0$ is isogenous to $\mathcal{A}_{s_i}$ for all $i =1,\hdots,m$. Let $\phi_i = \phi_{s_i}: A_0 \to \mathcal{A}_{s_i}$ be the isogeny chosen above (with $\psi_i = \psi_{s_i}$ and $d_i = d_{s_i}$) and let $X_i$ be an irreducible component of the projection to the second factor of $\mathcal{V}_{s_i} = \mathcal{V} \cap (\{s_i\} \times \mathbb{P}^{R}) \subset \mathcal{A}_{s_i}$ ($i=1,\hdots,m$), where we identify $\mathcal{V}$ and $\mathcal{A}$ with their images in $\mathbb{P}^1 \timesÊ\mathbb{P}^R$. It follows from $\dim \mathcal{V} = m$, $\pi(\mathcal{V}) = S$, and the Fiber Dimension Theorem (Corollary 14.116 in \cite{MR2675155}) that $\dim X_i = m-1$.

If $H$ denotes a hyperplane in $\mathbb{P}^{R}$, $H'$ denotes a hyperplane (i.e. a point) in $\mathbb{P}^1$, $\overline{\mathcal{V}}$ denotes the Zariski closure of $\mathcal{V}$ in $\mathbb{P}^1Ê\times \mathbb{P}^R$, and the class of a cycle $C$ modulo numerical equivalence is denoted by $[C]$, the degrees of the $X_i$ (as subvarieties of $\mathbb{P}^{R}$) can be estimated as
\begin{equation}\label{eq:celsiusorfahrenheit}
\deg X_i = [\{s_i\} \times X_i] \cdot [\mathbb{P}^{1}Ê\times H]^{m-1} \leq [\overline{\mathcal{V}}] \cdot [H' \times \mathbb{P}^{R}] \cdot [\mathbb{P}^{1}Ê\times H]^{m-1}
\end{equation}
since every irreducible component of the intersection of $\overline{\mathcal{V}}$ with the hyperplane $\{s_i\} \times \mathbb{P}^{R}$ is of dimension $m-1$, $\{s_i\}Ê\times X_i$ is an irreducible component of this intersection (of multiplicity at least $1$), and the other irreducible components of the intersection contribute non-negatively to the intersection product by Proposition 8.2 in \cite{MR732620}. Note that the right-hand side is independent of $s_i$, so $\deg X_i$ is uniformly bounded.

Let $x_i \in X_{i} \cap \mathcal{A}_{\Gamma}$ be arbitrary points such that $x_i$ does not lie in a translate of a positive-dimensional abelian subvariety of $\mathcal{A}_{s_i}$ contained in $\mathcal{V}_{s_i}$. Here and in the following, we identify $\mathcal{V}_{s_i}$ and $\mathcal{A}_{s_i}$ with their images in $\mathbb{P}^R$ so that $X_i \subset \mathcal{V}_{s_i} \subset \mathcal{A}_{s_i} \subset \mathbb{P}^R$ ($i=1,\hdots,m$).

We set $\zeta_i = \tilde{\phi}_i(x_i) \in \Gamma$, where $\tilde{\phi}_i = (\tilde{\psi}_i,\hdots,\tilde{\psi}_i,d_i \cdot \id_{E_1\times\cdots\times E_{g'}})$ and $\tilde{\psi}_i: \mathcal{E}_{s_i} \to E_0$ is the isogeny satisfying $\tilde{\psi}_i \circ \psi_i = (\deg \psi_i) \cdot \id_{E_0}$. We record the sequence of inequalities

\begin{multline}\label{eq:sandwich}
d_i^2\widehat{h}_{s_i}(x_i) \leq  \sum_{j=1}^{g-g'}{\widehat{h}_{E_0}(\tilde{\psi}_{i}(\pi_j(x_{i})))} +Ê\sum_{j=1}^{g'}{d_i^2\widehat{h}_{E_j}(\pi_{g-g'+j}(x_i))} = \widehat{h}_{A_0}(\zeta_i)\\
\leq (\deg \psi_i)\widehat{h}_{s_i}(x_i) \leq 4d_i^2\widehat{h}_{s_i}(x_i),
\end{multline}
which follows from
\[Ê\widehat{h}_{E_0}(\tilde{\psi}_{i}(\pi_j(x_i))) = (\deg \tilde{\psi}_{i})\widehat{h}^0_{s_i}(\pi_j(x_i)) = (\deg \psi_{i})\widehat{h}^0_{s_i}(\pi_j(x_i))\]
for $j=1,\hdots,g-g'$.

Assuming that Theorem \ref{thm:vojtaheightbound} is false, we aim to deduce a contradiction with the generalized Vojta-R\'emond inequality in Theorem \ref{thm:vojta} applied to the $x_i$ and $X_i$ ($i=1,\hdots,m$). In the following, we show how to choose the various objects and parameters in the generalized Vojta-R\'emond inequality in order to arrive at such a contradiction. For $i \in \{1,\hdots,m\}$, we choose the very ample line bundle $\mathcal{L}_i$ on $X_i$ from Theorem \ref{thm:vojta} to be the restriction of $\mathcal{L}_{s_i}$ to $X_i$ and the associated system of homogeneous coordinates $W^{(i)}$ to be the one induced by the closed embedding $X_i \hookrightarrow \mathbb{P}^R$ (so $N_i = R$). We have $X = X_1 \times \cdots \times X_m$, $x = (x_1,\hdots,x_m)$, and $u_0 = m(m-1)$.

We define $\lVert \gamma \rVert = \sqrt{\widehat{h}_{A_0}(\gamma)}$ for $\gamma \in A_0$, which extends to a norm on $A_0 \otimes \mathbb{R}$. By fundamental properties of the N\'{e}ron-Tate height, there exists a constant $c_{A_0} > 0$, depending only on $A_0$ and its embedding into $\mathbb{P}^R$, such that
\begin{equation}\label{eq:heightcomparisonazero}
\left|\widehat{h}_{A_0}(\gamma)-h_{A_0}(\gamma)\right| \leq c_{A_0}
\end{equation}
for all $\gamma \in A_0$.

\begin{lem}\label{lem:heightcomparison}
Let $s \in S$ and $p \in \mathcal{A}_s$. There exists a constant $\tilde{c}_2$, depending only on $\mathcal{A}$ and its quasi-projective immersion, such that
\[ |h_s(p)-\widehat{h}_s(p)| \leq \tilde{c}_2\max\{h_{\overline{S}}(s),1\}.\]
\end{lem}

\begin{proof}
See \cite{MR0419455} and note that $\mathcal{E}_s$ is given (in $\mathbb{P}^2$) by an equation $y^2z = \tilde{x}^3 - A(s)\tilde{x}z^2 - B(s)z^3$ with $A(s)=\frac{1}{3}(s^2-s+1)$, $B(s)=\frac{1}{27}(s+1)(s-2)(2s-1)$, and $\tilde{x} = x - \frac{s+1}{3}z$ and that the heights of $A(s)$ and $B(s)$ are bounded by a constant multiple of $\max\{h_{\overline{S}}(s),1\}$ (independently of $s$).
\end{proof}

Let $F$ be a non-zero linear form on $\mathbb{P}^1$ that vanishes at $s_i$. Using the theory of heights of subvarieties of multiprojective spaces (see Section \ref{sec:preliminaries} and \cite{MR1837829}, whose notation we use), we can estimate the height of $X_i$ as a subvariety of $\mathbb{P}^{R}$ thanks to Corollaire 2.4 in \cite{MR1837829} as
\[h(X_{i}) = h_{(0,\dim X_i+1)}\left(\{s_i\}Ê\times X_i\right) \leq h_{(0,\dim \mathcal{V})}((\{s_i\} \timesÊ\mathbb{P}^{R}) \cap \bar{\mathcal{V}}).\]

We apply R\'{e}mond's multiprojective version of the arithmetic B\'{e}zout theorem for the special case of an intersection with a multiprojective ``hypersurface" (Th\'{e}or\`{e}me 3.4 in \cite{MR1837829}) and obtain
\[ h_{(0,\dim \mathcal{V})}((\{s_i\} \timesÊ\mathbb{P}^{R}) \cap \bar{\mathcal{V}})\leq h_{(1,\dim \mathcal{V})}(\bar{\mathcal{V}})+ d_{(0,\dim \mathcal{V})}(\bar{\mathcal{V}})h_{\bar{\mathcal{V}},(0,\dim \mathcal{V})}(F).\]

By applying Corollaire 3.6 and Lemme 3.3 from \cite{MR1837829}, we may bound $h_{\bar{\mathcal{V}},(0,\dim \mathcal{V})}(F)$ from above by $h(F)+\frac{\log 2}{2}$ (the height of a form used here is defined as in Paragraph 2.1 of \cite{MR1837829}). It is elementary that $h(F)$ is bounded from above linearly in terms of $\max\{h_{\overline{S}}(s_i),1\}$ (with an absolute constant) and so we obtain that
\begin{equation}\label{eq:arithmeticbezout}
h(X_i) \leq \tilde{c}_3\max\{h_{\overline{S}}(s_i),1\}.
\end{equation}

We have constants $c_1 = c_2 = \Lambda^{\psi(0)}$ from the generalized Vojta-R\'emond inequality in Theorem \ref{thm:vojta}, which can be bounded from above independently of the $s_i$ (in fact, we have $c_1 \preceq 1$ as we will see, when fixing the parameters $\theta$, $\omega$, $M$, $t_1$, $t_2$, $\delta_1$, \dots, $\delta_m$).

We assume now that
\begin{equation}\label{eq:growingfast}
\lVert \zeta_{i+1} \rVert\geq 4d_{i+1}\sqrt{c_2}\lVert \zeta_i \rVertÊ\quad (i=1,\hdots,m-1)
\end{equation}
and that
\begin{equation}\label{eq:heightboundpointone}
h_{s_i}(x_i) > 8c_1(\tilde{c}_2\max\{h_{\overline{S}}(s_i),1\}+c_{A_0}) \quad (i=1,\hdots,m)
\end{equation}
as well as
\begin{equation}\label{eq:heightboundpoint}
\lVert \zeta_i \rVert^2 \geq 8d_i^2\Lambda^{2\psi(0)}(Mt_2)^{u_0}(\tilde{c}_3\max\{h_{\overline{S}}(s_i),1\}+\delta_i) \quad (i=1,\hdots,m).
\end{equation}
We will deduce a contradiction from this together with the condition that the $\zeta_i$ lie in a cone of small angle in $\Gamma \otimes \mathbb{R}$, which will imply an ineffective height bound of the desired form thanks to \eqref{eq:masserwuestholz} and \eqref{eq:sandwich}, thereby proving Theorem \ref{thm:vojtaheightbound}. Of course, the bound depends on the choice of parameters in the generalized Vojta-R\'emond inequality, which we will fix later.

We define recursively $b_{m}=1$ and
\begin{equation}\label{eq:recursivedefinition}
b_{i-1} = \left[\frac{b_i \lVert \zeta_i \rVert}{\lVert \zeta_{i-1} \rVert}\right]+1 \geq \sqrt{c_2}b_id_i \quad (i=2,\hdots,m),
\end{equation}
where the lower bound follows from \eqref{eq:growingfast}. We then set $a_i = 4(b_id_i)^2$. The generalized Vojta-R\'emond inequality yields additional constants
\[c^{(i)}_3 = \Lambda^{2\psi(0)}(Mt_2)^{u_0}(h(X_i)+\delta_i) \quad (i=1,\hdots,m)\]
(depending on the parameters $\theta$, $\omega$, $M$, $t_1$, $t_2$, $\delta_1$, \dots, $\delta_m$, which will be chosen later).

It follows from \eqref{eq:recursivedefinition} that
\begin{equation}\label{eq:fastgrowth}
a_{i-1} \geq c_2 a_i \quad (i=2,\hdots,m).
\end{equation}

We can estimate
\[ h_{s_i}(x_i) \geq \frac{1}{2}(h_{s_i}(x_i)+\tilde{c}_2\max\{h_{\overline{S}}(s_i),1\}) \geq \frac{1}{2}\widehat{h}_{s_i}(x_i)\]
thanks to Lemma \ref{lem:heightcomparison} and \eqref{eq:heightboundpointone}. We know from \eqref{eq:arithmeticbezout} that
\[ c^{(i)}_3 \leq \Lambda^{2\psi(0)}(Mt_2)^{u_0}(\tilde{c}_3\max\{h_{\overline{S}}(s_i),1\}+\delta_i).\]
It then follows from \eqref{eq:sandwich} and \eqref{eq:heightboundpoint} that
\begin{equation}\label{eq:largeheight}
c^{(i)}_3 \leq \frac{1}{8d_i^2}\lVert \zeta_i \rVert^2 \leq \frac{1}{2}\widehat{h}_{s_i}(x_i) \leq h_{s_i}(x_i)Ê\quad (i=1,\hdots,m).
\end{equation}

Assume now that
\begin{equation}\label{eq:littleangle}
\langle \zeta_i, \zeta_{i+1} \rangle \geq \left(1-\frac{1}{32c_1}\right)\lVert \zeta_i \rVertÊ\lVertÊ\zeta_{i+1} \rVert,
\end{equation}
where $\langle \cdot, \cdot \rangle$ denotes the scalar product associated to $\lVert \cdot \rVert$ ($i=1,\hdots,m-1$). This means that the angle between $\zeta_i$ and $\zeta_{i+1}$ with respect to $\langle \cdot, \cdot \rangle$ is small. Since $\GammaÊ\otimes \mathbb{R}$ is a finite-dimensional Euclidean vector space with respect to $\lVertÊ\cdotÊ\rVert$, this partitions the points $p \in \mathcal{V} \cap \mathcal{A}_\Gamma$ that do not lie in a translate of a positive-dimensional abelian subvariety of $\mathcal{A}_{\pi(p)}$ contained in $\mathcal{V}_{\pi(p)}$ into a certain finite number of sets that depends only on $c_1$ and $\Gamma$. After fixing the parameters in the generalized Vojta-R\'emond inequality, we will see that $c_1 \preceq 1$ and hence the number of sets can be bounded from above similarly (and independently of the $s_i$).
 
We aim to reach a contradiction. We have
\[0 \leq \frac{b_i}{b_{i+1}}-\frac{\lVert \zeta_{i+1} \rVert}{\lVert \zeta_{i} \rVert} \leq 1\]
by construction and
\[ \left\lVert \frac{\zeta_i}{\lVert \zeta_{i} \rVert} - \frac{\zeta_{i+1}}{\lVert \zeta_{i+1} \rVert} \right\rVert \leq (4\sqrt{c_1})^{-1} \]
by our assumption on the angle.

It follows from the triangle inequality for $\lVert \cdotÊ\rVert$ that
\[ \lVert b_i \zeta_i - b_{i+1}\zeta_{i+1}\rVert \leq \left(b_i-\frac{b_{i+1}\lVert\zeta_{i+1}\rVert}{\lVert\zeta_i\rVert}\right)\lVert\zeta_i\rVert + b_{i+1}\lVert \zeta_{i+1} \rVert\left\lVert\frac{\zeta_i}{\lVert \zeta_{i}\rVert} - \frac{\zeta_{i+1}}{\lVert \zeta_{i+1}\rVert} \right\rVert,\]
which implies together with the above inequalities that
\[\lVert b_i \zeta_i - b_{i+1}\zeta_{i+1} \rVert \leq b_{i+1}\lVert \zeta_i \rVert+\frac{b_{i+1}}{4\sqrt{c_1}}\lVert \zeta_{i+1} \rVert.\]
Using that $\lVert \zeta_{i+1} \rVert \geq 4\sqrt{c_2} \lVert \zeta_i \rVert = 4\sqrt{c_1} \lVert \zeta_i \rVert$ by \eqref{eq:growingfast}, we can conclude that
\[ \lVert b_i \zeta_i - b_{i+1}\zeta_{i+1} \rVert \leq \frac{b_{i+1}\lVert \zeta_{i+1} \rVert}{2\sqrt{c_1}}.\]

This implies thanks to \eqref{eq:sandwich} that
\[Ê4c_1\widehat{h}_{A_0}(b_i\zeta_i-b_{i+1}\zeta_{i+1}) \leq b_{i+1}^2\widehat{h}_{A_0}(\zeta_{i+1}) \leq 4(d_{i+1}b_{i+1})^2\widehat{h}_{s_{i+1}}(x_{i+1})\]
and thus
\[ c_1\widehat{h}_{A_0}(b_i\zeta_i-b_{i+1}\zeta_{i+1}) \leq \frac{1}{4}a_{i+1}\widehat{h}_{s_{i+1}}(x_{i+1}).\]
We have already seen that $\widehat{h}_{s_{i+1}}(x_{i+1}) \leq 2h_{s_{i+1}}(x_{i+1})$ and hence
\[Êc_1\widehat{h}_{A_0}(b_i\zeta_i-b_{i+1}\zeta_{i+1}) \leq \frac{1}{2}a_{i+1}h_{s_{i+1}}(x_{i+1}).\]

We rewrite the left-hand side using the fact that $\widehat{h}_{A_0}$ satisfies the parallelogram law and get
\[Êc_1\left(2b_i^2\widehat{h}_{A_0}(\zeta_i)+2b_{i+1}^2\widehat{h}_{A_0}(\zeta_{i+1})-\widehat{h}_{A_0}(b_i\zeta_i+b_{i+1}\zeta_{i+1})\right) \leq \frac{1}{2}a_{i+1}h_{s_{i+1}}(x_{i+1}).\]

Using \eqref{eq:sandwich}, we deduce that
\begin{align*}
c_1\Bigg(\sum_{j=1}^{g-g'}{\left(2b_i^2(\deg \tilde{\psi}_i)\widehat{h}^0_{s_i}(\pi_j(x_i))+2b_{i+1}^2(\degÊ\tilde{\psi}_{i+1})\widehat{h}^0_{s_{i+1}}(\pi_j(x_{i+1}))\right)}\\
+\sum_{j=1}^{g'}{\left(2b_i^2d_i^2\widehat{h}_{E_j}(\pi_{g-g'+j}(x_i))+2b_{i+1}^2d_{i+1}^2\widehat{h}_{E_j}(\pi_{g-g'+j}(x_{i+1}))\right)}\\
-\widehat{h}_{A_0}(b_i\zeta_i+b_{i+1}\zeta_{i+1})\Bigg) \leq \frac{1}{2}a_{i+1}h_{s_{i+1}}(x_{i+1}).
\end{align*}
Recall that by abuse of notation, we use $\pi_j$ for the projection onto the $j$-th factor in any fiber of $\mathcal{A} \to S$ ($j=1,\hdots,g$).

Using \eqref{eq:heightcomparisonazero}, \eqref{eq:heightboundpointone}, and Lemma \ref{lem:heightcomparison}, we obtain by adding up that
\begin{multline}\label{eq:dafundamentalinequality}
c_1\left( \sum_{j=1}^{g-g'}{\left(2b_1^2(\deg \tilde{\psi}_1){h}^0_{s_1}(\pi_j(x_1))+2b_{m}^2(\degÊ\tilde{\psi}_m){h}^0_{s_{m}}(\pi_j(x_{m}))\right)}\right.\\
+\sum_{j=1}^{g'}{\left(\frac{a_1}{2}h_{E_j}(\pi_{g-g'+j}(x_1))+\frac{a_m}{2}h_{E_j}(\pi_{g-g'+j}(x_{m}))\right)}\\
+\sum_{i=2}^{m-1}\left(\sum_{j=1}^{g-g'}{4b_i^2(\deg \tilde{\psi}_i){h}^0_{s_i}(\pi_j(x_i))}+\sum_{j=1}^{g'}{a_ih_{E_j}(\pi_{g-g'+j}(x_i))}\right)\\
\left. -\sum_{i=1}^{m-1}{h_{A_0}\left(b_{i}\zeta_{i}+b_{i+1}\zeta_{i+1}\right)}\right) \leq \frac{1}{2}\sum_{i=2}^{m}{a_{i}h_{s_{i}}\left(x_{i}\right)}Ê\\
+ \sum_{i=1}^{m}{(4a_{i}c_1(\tilde{c}_2\max\{h_{\overline{S}}(s_i),1\}+c_{A_0}))} < \sum_{i=1}^{m}{a_{i} h_{s_{i}}\left(x_{i}\right)}.
\end{multline}
Note that Lemma \ref{lem:heightcomparison} implies that
\[Ê\sum_{j=1}^{g-g'}{h^{0}_{s_i}(\pi_{j}(x_i))}Ê\leq \sum_{j=1}^{g-g'}{\widehat{h}^0_{s_i}(\pi_{j}(x_i))}+\tilde{c}_2\max\{h_{\overline{S}}(s_i),1\}\]
if we take $p = (\pi_1(x_i),\hdots,\pi_{g-g'}(x_i),0_{E_1},\hdots,0_{E_{g'}})$.

Recall that $X = X_1 \times \cdotsÊ\times X_m$. We consider the morphism $\Psi: X \to A_0^{m-1}$ given by
\[(y_1,\hdots,y_{m}) \mapsto (b_1\tilde{\phi}_{1}(y_1)-b_2\tilde{\phi}_{2}(y_2),\hdots,b_{m-1}\tilde{\phi}_{m-1}(y_{m-1})-b_{m}\tilde{\phi}_{m}(y_{m})).\]
If $p_1,\hdots,p_{m}$, $q_1,\hdots,q_{m-1}$ are the natural projections on $X$ and $A_0^{m-1}$ respectively, we obtain two line bundles on $X$, namely
\[ \mathcal{N}_a = p_1^{\ast}\mathcal{L}_{1}^{\otimes a_1} \otimes \cdots \otimes p_{m}^{\ast}\mathcal{L}_{m}^{\otimes a_{m}}\]
and
\[ \mathcal{M} = \Psi^{\ast}(q_1^{\ast}L_0 \otimes \cdots \otimes q_{m-1}^{\ast}L_0).\]
Recall that $\mathcal{L}_i$ is the restriction of $\mathcal{L}_{s_i}$ to $X_i$ and that the system of homogeneous coordinates $W^{(i)}$ for $\mathcal{L}_{i}$ is the one induced by the closed embedding $X_i \hookrightarrow \mathbb{P}^R$. We identify $W^{(i)}$ with $p_i^{\ast} W^{(i)}$ ($i=1,\hdots,m$).

We let $\sigma: A_0 \times A_0 \to A_0$ and $\delta: A_0 \times A_0 \to A_0$ be the sum and difference morphism respectively. We have $q_i \circÊ\Psi = \delta \circ ([b_i],[b_{i+1}]) \circ (\tilde{\phi}_i|_{X_i},\tilde{\phi}_{i+1}|_{X_{i+1}}) \circ (p_i,p_{i+1})$, where $[b]$ denotes the multiplication-by-$b$ morphism on $A_0$ for $b \in \mathbb{Z}$. Since $L_0$ is symmetric, we have by Proposition A.7.3.3 in \cite{MR1745599} that
\[Ê\delta^{\ast}L_0 \simeq \pr_1^{\ast}L_0^{\otimes 2}Ê\otimes \pr_2^{\ast}L_0^{\otimes 2}Ê\otimes \sigma^{\ast}L_0^{\otimes(-1)},\]
where $\pr_i: A_0 \times A_0 \to A_0$ ($i=1,2$) are the natural projections.

It follows that
\[Ê\mathcal{M} \simeq \mathcal{P} \otimes \left(\tilde{\Psi}^{\ast}(q_1^{\ast}L_0 \otimes \cdots \otimes q_{m-1}^{\ast}L_0)\right)^{\otimes-1},\]
where 
\[\mathcal{P} = (\tilde{\phi}_{1}|_{X_1} \circ p_1)^{\ast}\left(L_0^{\otimes 2b_1^2}\right) \otimes \bigotimes_{i=2}^{m-1}{(\tilde{\phi}_i|_{X_i} \circ p_i)^{\ast}\left(L_0^{\otimes 4b_i^2}\right)} \otimes (\tilde{\phi}_{m}|_{X_m} \circ p_m)^{\ast}\left(L_0^{\otimes 2b_m^2}\right)\]
and $\tilde{\Psi}$ is defined by replacing every minus in the definition of $\Psi$ with a plus. By our choice of isogenies, the line bundle $\mathcal{P}$ is isomorphic to the very ample line bundle associated to the composition $\iota'$ of the closed embedding $X \hookrightarrow \mathcal{A}_{s_1}Ê\timesÊ\cdots \times \mathcal{A}_{s_m} \hookrightarrow (\mathbb{P}^5)^{gm}$ with Veronese embeddings of each $\mathbb{P}^5$ of degrees $\alpha_ib_i^2d_i^2$ and $\alpha_ib_i^2\degÊ\psi_i$ ($i=1,\hdots,m$), where $\alpha_i = 2$ if $i \in \{1,m\}$ and $4$ otherwise, and finally the Segre embedding.

It follows from $\degÊ\psi_i \leq 4d_i^2$ that $\mathcal{P}$ injects into $\mathcal{N}_a^{\otimes 4}$, so we can set $t_1=4$. The embedding $\iota'$ yields a system of homogeneous coordinates $\Xi$ for $\mathcal{P}$ and the injection can be chosen such that they map to monomials in the $W^{(i)}$ in $\Gamma(X,\mathcal{N}_a^{\otimes 4})$.

The line bundle $\mathcal{M}$ is nef as it is a tensor product of pull-backs of nef line bundles under proper morphisms and we have
\[Ê\mathcal{P} \otimes \mathcal{M}^{\otimes-1} \simeq \tilde{\Psi}^{\ast}(q_1^{\ast}L_0 \otimes \cdots \otimes q_{m-1}^{\ast}L_0).\]

We will see in Lemma \ref{lem:isogenycontroltwo} that the line bundle on the right injects into $\mathcal{N}_a^{\otimes 8}$ (so we choose $t_2 = 8$). In the same lemma, we show that this line bundle is generated by a set of sections $Z$ of cardinality $M = (R+1)^{m-1}$ satisfying
\[Êh(Z(x)) = \sum_{i=1}^{m-1}{h_{A_0}\left(b_{i}\zeta_{i}+b_{i+1}\zeta_{i+1}\right)}.\]
Furthermore, these sections correspond under the given injection to polynomials of multidegree $t_2 a$ in the projective coordinates $W^{(i)}$ on $X$ of height bounded by $\sum_{i=1}^{m}{a_i\delta_i}$, where each $\delta_i$ is bounded polynomially in $d_i$ and $h_{\overline{S}}(s_i)$.

It should be noted here that the mentioned injection is achieved by writing $\mathcal{N}_a^{\otimes t_2}$ as the tensor product of $\mathcal{P}Ê\otimes \mathcal{M}^{\otimes(-1)}$ with a globally generated line bundle and then sending each section to its tensor product with a global section of this second line bundle that does not vanish at $x$. Thus, the injection depends on $x$, but the number of choices for the injection can be bounded independently of $x$. The set $Z$ is always the same but corresponds to different multihomogeneous polynomials depending on the chosen injection. However, the height of the polynomials is bounded independently of the choice of injection. Analogous statements hold for the injection of $\mathcal{P}$ into $\mathcal{N}_a^{\otimes t_1}$ and $\Xi$.

With $h_{\mathcal{M}}(x) = h(\Xi(x))-h(Z(x))$, the inequality \eqref{eq:dafundamentalinequality} then amounts to
\begin{equation}\label{eq:superbadinequality}
c_1h_{\mathcal{M}}(x) < h_{\mathcal{N}_a}(x),
\end{equation}
which yields the desired contradiction with Theorem \ref{thm:vojta} provided that we can choose the parameters in Theorem \ref{thm:vojta} such that all conditions of Theorem \ref{thm:vojta} are satisfied and such that the parameters are bounded in the required way. We will achieve this in the following lemmata.

First of all, we set $\mathcal{X} = X$ and $\pi = \id$. Consequently, we have $\mathcal{Y} = Y$ for every subproduct $Y \subset X$ in Theorem \ref{thm:vojta}. We set $\omega = 0$, which will be justified by Lemma \ref{lem:voidbringer}. Before stating and proving this lemma (and Lemma \ref{lem:findus} from which it follows), we have to recall some terminology: By a cycle on $\mathcal{A}_s$ ($s \in S$), we mean a formal sum of irreducible subvarieties of $\mathcal{A}_s$ (with integer coefficients). If $C$ is a cycle on $\mathcal{A}_s$, we denote by $[C]$ its equivalence class modulo numerical equivalence. We will also call $[C]$ the class of $C$ for short. An effective class is a class a positive integral multiple of which contains an effective cycle (i.e. a formal sum of irreducible subvarieties with non-negative integer coefficients). For a natural number $n$ and a cycle $C$, we denote by $n[C]$ the class $[C+C+\cdots+C]$ ($n$ times), where the addition takes place in the additive group of cycles and not on the abelian variety.

\begin{lem}\label{lem:findus}
There exists a natural number $N$, depending only on $g$, such that for every $s \in S$ such that $\mathcal{E}_s$ is isogenous to $E_0$ and for every cycle $Y$ on $\mathcal{A}_s$, the class $N[Y]$ lies in the subring of the ring of cycles modulo numerical equivalence that is generated under addition and intersection product by $\mathcal{E}_s^{g-g'}Ê\times E_1 \timesÊ\cdotsÊ\times E_{g'}$ itself together with the hypersurfaces defined by $\pi_j(z) = 0_{\mathcal{E}_s}$ ($j=1,\hdots,g-g'$), $\pi_{g-g'+j}(z) = 0_{E_j}$ ($j=1,\hdots,g'$), and $\pi_j(z) = \pi_k(z)$ ($1 \leq j < k \leq g-g'$).
\end{lem}

In the proof of this lemma, we use the restrictions we placed on $\mathcal{A}$ and $A_0$ in a crucial way.

\begin{proof}
Let $s \in S$ such that $\mathcal{E}_s$ is isogenous to $E_0$ and let $Y$ be a cycle on $\mathcal{A}_s$. We assume that $Y$ is equidimensional of dimension $\dim Y$. Let us denote the subring mentioned in the lemma by $\mathfrak{D}$. We denote by $\mathfrak{D}_k$ the intersection of $\mathfrak{D}$ with the additive subgroup of all classes of equidimensional cycles of dimension $k$. If $\dim Y = g$, the assertion of the lemma is true with $N = 1$.

For the proof, we will use different equivalence relations on the set of cycles on $\mathcal{A}_s$, namely algebraic, homological, and numerical equivalence. For definitions, see Chapters 10 and 19 of \cite{MR732620}. In order to use homological equivalence, we will tacitly identify $\mathcal{A}_s$ with its base change to $\mathbb{C} \supset \bar{\mathbb{Q}}$; if the statement of the lemma holds over $\mathbb{C}$, it automatically holds over $\bar{\mathbb{Q}}$ as well. Note that the Borel-Moore homology used by Fulton coincides with the usual singular homology since the ambient variety $\mathcal{A}_s$ is projective. Also, since we are working on an abelian variety, numerical and homological equivalence coincide (see \cite{MR0230336}) and for codimension $1$ cycles all three equivalence relations coincide (see \cite{MR732620}, Section 19.3.1, and use the fact that the N\'eron-Severi group of an abelian variety is torsion-free).

There is a natural isomorphism between the additive group of cycles of codimension $1$ modulo algebraic equivalence and the N\'eron-Severi group of $\mathcal{A}_s$. We fix ample line bundles on $\mathcal{E}_s$ and $E_1$, \dots, $E_{g'}$ that induce principal polarizations on these elliptic curves. The tensor product of the pull-backs of these line bundles under the projections on the factors induces a principal polarization on $\mathcal{A}_s$. This polarization induces an isomorphism between the N\'eron-Severi group of $\mathcal{A}_s$ and the additive group of endomorphisms of $\mathcal{A}_s$ that are fixed by the corresponding Rosati involution, after tensoring both with $\mathbb{Q}$ (see \cite{MR861974}, Proposition 17.2). Since the polarization is principal, the proof of Proposition 17.2 in \cite{MR861974} shows that we also get an isomorphism without tensoring with $\mathbb{Q}$.

By our hypotheses on $\Hom(E_i,E_j)$ ($i,j = 0,\hdots, g'$), the endomorphism ring of $\mathcal{A}_s$ is naturally isomorphic to either $\M_{g-g'}(\mathbb{Z}) \times E$ or $\End(\mathcal{E}_s) \times E$, where $E = \prod_{i=1}^{g'}{\End(E_i)}$. In both cases, the set of endomorphisms that are fixed by the Rosati involution corresponds to $M_{g-g'}^{s}(\mathbb{Z}) \times \mathbb{Z}^{g'}$, where $M_{g-g'}^{s}(\mathbb{Z})$ denotes the set of symmetric $(g-g') \times (g-g')$-matrices with entries in $\mathbb{Z}$.

We fix a $\mathbb{Z}$-basis for this additive group by choosing the standard basis for $\mathbb{Z}^{g'}$ and the basis $\{A_i, A_{j,k}; i = 1,\hdots, g-g', 1Ê\leq j < k \leq g-g'\}$ for $M_{g-g'}^{s}(\mathbb{Z})$ with $A_i = (a^{i}_{r,s})_{1Ê\leq r,s \leq g-g'}$, $A_{j,k} = (a^{j,k}_{t,u})_{1Ê\leq t,u \leq g-g'}$, $a^{i}_{r,s} = 1$ if $r=s=i$ and $0$ otherwise, and $a^{j,k}_{t,u} = 1$ if $t = u \in \{j,k\}$, $-1$ if $(t,u) \in \{(j,k),(k,j)\}$, and $0$ otherwise.

For a line bundle $L$ on $\mathcal{A}_s$, we denote the associated homomorphism from $\mathcal{A}_s$ to $\widehat{\mathcal{A}}_s$ by $\phi_L: \mathcal{A}_s \to \widehat{\mathcal{A}}_s$. If $\widetilde{\mathcal{L}_s}$ is the ample line bundle that induces the principal polarization on $\mathcal{A}_s$ and $\chi \in \End(\mathcal{A}_s)$ is fixed by the corresponding Rosati involution, then we have $\phi_{\chi^{\ast}\widetilde{\mathcal{L}_s}} = \phi_{\widetilde{\mathcal{L}_s}} \circÊ\chi \circ \chi$. Using this fact, we can compute that under the above isomorphism the chosen basis of $M_{g-g'}^{s}(\mathbb{Z}) \times \mathbb{Z}^{g'}$ corresponds precisely to the collection of hypersurfaces that generate $\mathfrak{D}$ (modulo algebraic equivalence). This proves the assertion of the lemma in codimension $1$ with $N = 1$.

Thanks to Murty, who showed in \cite{MR1104704} that (after tensoring with $\mathbb{Q}$) any cycle on a product of elliptic curves over $\mathbb{C}$ is homologically (and hence numerically) equivalent to a $\mathbb{Q}$-linear combination of intersections of divisors, we then know that $[Y]$ lies in $\mathbb{Q}\mathfrak{D}_{\dim Y} = \mathfrak{D}_{\dim Y}Ê\otimes_{\mathbb{Z}} \mathbb{Q}$. The intersection product yields (by definition of numerical equivalence) a non-degenerate bilinear form $\mathbb{Q}\mathfrak{D}_{\dim Y} \timesÊ\mathbb{Q}\mathfrak{D}_{g-\dim Y} \to \mathbb{Q}$, which we will denote by $\langle \cdot, \cdot \rangle$. In particular, we have $d := \dim \mathbb{Q}\mathfrak{D}_{\dim Y} = \dim \mathbb{Q}\mathfrak{D}_{g-\dim Y}$.

We choose a $\mathbb{Q}$-basis $(v_1,\hdots,v_d)$ for $\mathbb{Q}\mathfrak{D}_{\dim Y}$ and a $\mathbb{Q}$-basis $(w_1,\hdots,w_d)$ for $\mathbb{Q}\mathfrak{D}_{g-\dim Y}$, both consisting of intersections of the hypersurfaces described in the lemma. Note that the intersection product of $v_i$ and $w_j$ is either $0$ or $1$ for all $i$ and $j$ since any collection of $g$ hypersurfaces as in the lemma either meets in a positive-dimensional component (so does not meet at all after translating one of the hypersurfaces by a sufficiently generic point) or meets transversely in the origin of $\mathcal{A}_s$.

We can write $[Y] = \sum_{i=1}^{d}{\lambda_i v_i}$ with $\lambda_i \in \mathbb{Q}$ ($i=1,\hdots,d$). The intersection product of $Y$ with each of the $w_j$ is an integer, so we get that $\sum_{i=1}^{d}{\langle v_i, w_j \rangle \lambda_i} \in \mathbb{Z}$ ($j=1,\hdots,d$). We conclude that $\Delta\lambda_i \in \mathbb{Z}$ ($i=1,\hdots,d$), where $\Delta$ is the (non-zero) determinant of the matrix $(\langle v_i, w_jÊ\rangle)_{i,j=1,\hdots,d}$. Now $d$ can be bounded in terms of $g$ and then $|\Delta|$ can be bounded in terms of $g$ by Hadamard's determinant inequality -- note that each entry of the matrix is either $0$ or $1$. By taking the least common multiple of all possible $\Delta$, we obtain a natural number $N$ such that $N[Y] \in \mathcal{D}_{\dim Y}$ and $N$ depends only on $g$. We can now take $N$ to be the least common multiple of all the $N$ we obtain for varying $\dim Y \in \{0, 1,\hdots,g\}$ and the lemma follows.
\end{proof}

\begin{lem}\label{lem:voidbringer}
Suppose that $x_i$ does not lie in a translate of a positive-dimensional abelian subvariety of $\mathcal{A}_{s_i}$ that is contained in $X_i$ ($i=1,\hdots,m$). There exists an integer $\theta \geq 1$, depending only on $g$ and $m$, but independent of $x$, $X$, and $Y$, such that
\[Ê\mathcal{M}^{\cdot \dim(Y)} \cdot Y \geq \theta^{-1}\prod_{i=1}^m{a_i^{\dim(Y_i)}}\]
for every subproduct $Y = Y_1 \times \cdots \times Y_m \subset X$ such that $Y_i \subset X_i$ is an irreducible subvariety ($i = 1, \hdots, m$) and $x \in Y$.
\end{lem}

\begin{proof}
We define a finite morphism $\Phi: X \to A_0^{m}$ by
\[\Phi(y_1,\hdots,y_{m}) = (\tilde{\phi}_{1}(y_1),\hdots,\tilde{\phi}_{m}(y_{m})).\]
The morphism $\Psi$ factorizes as $\Psi' \circÊ\Phi$ with
\[\Psi'(y_1,\hdots,y_m) = (b_1y_1-b_2y_2,\hdots,b_{m-1}y_{m-1}-b_{m}y_m)\]
and we get a line bundle $\tilde{\mathcal{M}} = \Psi'^{\ast}(q_1^{\ast}L_0 \otimes \cdots \otimes q_{m-1}^{\ast}L_0)$ on $A_0^{m}$ such that $\mathcal{M} = \Phi^{\ast}\tilde{\mathcal{M}}$.

It follows from the projection formula that
\begin{equation}\label{eq:pushpull}
\mathcal{M}^{\cdot \dim(Y)} \cdot [Y] = \tilde{\mathcal{M}}^{\cdot \dim(Y)} \cdot \Phi_{\ast}([Y]).
\end{equation}
By a crucial homogeneity result of Faltings (see \cite{MR1109353}, Lemma 4.2, or \cite{MR1338861}, Corollary 11.4), we have
\begin{equation}\label{eq:intersectionfaltings}
\tilde{\mathcal{M}}^{\cdot \dim(\Phi(Y))} \cdot \Phi_{\ast}([Y]) = \left(\prod_{i=1}^{m}{b_i^{2\dim Y_i}}\right)(\mathcal{M}_1^{\cdot \dim(\Phi(Y))} \cdot \Phi_{\ast}([Y])),
\end{equation}
where $\mathcal{M}_1 = \Psi_1^{\ast}(q_1^{\ast}L_0 \otimes \cdots \otimes q_{m-1}^{\ast}L_0)$ and
\[\Psi_1(y_1,\hdots,y_m) = (y_1-y_2,\hdots,y_{m-1}-y_m).\]

We will now show that 
\[\mathcal{M}_1^{\cdot \dim(\Phi(Y))} \cdot \Phi_{\ast}([Y])Ê\geq \tilde{\theta}^{-1} \prod_{i=1}^{m}{d_i^{2\dim Y_i}}\]
for some integer $\tilde{\theta} \geq 1$, depending only on $g$ and $m$.

Let $N$ be the natural number furnished by Lemma \ref{lem:findus} and let $i \in \{1,\hdots,m\}$. The cycle $NY_i$ on $\mathcal{A}_{s_i}$ is numerically equivalent to a sum $\sum_{IÊ\subset \{1,\hdots,g'\}}{C_I \times C'_I}$ by Lemma \ref{lem:findus}, where $C_I$ is a $\mathbb{Z}$-linear combination of cycles on $\mathcal{E}_{s_i}^{g-g'}$, each given by a collection of equations of the form $\pi_j(z) = 0_{\mathcal{E}_{s_i}}$ or $\pi_j(z) = \pi_k(z)$ ($j \neq k$), and $C'_I =Ê\{(z_1,\hdots,z_{g'})Ê\in E_1 \times \cdotsÊ\times E_{g'}; z_j = 0_{E_j} \forall j \in I \}$. Furthermore, we can assume that $C_I$ is either zero or equidimensional of dimension $\dim Y_i -Ê\dim C'_I \geq 0$. It follows that
\[ [C_I] = N(\pi_1,\hdots,\pi_{g-g'})_{\ast}\left([Y_i] \cdot \left[\mathcal{E}_{s_i}^{g-g'} \times C'_{\{1,\hdots,g'\}\backslash I}\right]\right).\]

On an abelian variety, the intersection of two effective classes is again an effective class since $C$ and $u+C$ are algebraically equivalent for any element $u$ of the abelian variety and any irreducible subvariety $C$; if $C$ and $D$ are two irreducible subvarieties, then $u+C$ will intersect $D$ dimensionally transversely for all $u$ in an open Zariski dense set by the Fiber Dimension Theorem (Corollary 14.116 in \cite{MR2675155}). Furthermore, the push-forward of an effective class is always an effective class. So $[C_I]$ and hence $[C_I \times C'_I]$ is an effective class.

Here and in the following we implicitly use that the cartesian product with a numerically trivial cycle stays numerically trivial -- this follows from the fact that the class of a cartesian product is the intersection product of the pull-backs of the classes of its factors and the fact that numerical equivalence is preserved under pull-back with respect to a flat morphism between non-singular complete varieties (see Example 19.1.6 in \cite{MR732620}).

Using the special form of $C_I$ and $C'_I$, we compute that if $C_I$ is non-zero, then $\left(\tilde{\phi}_i\right)_{\ast}([C_I \times C'_I]) = (\degÊ\tilde{\psi}_i)^{\dim C_I}d_i^{2\dim C'_I}[E_I]$, where $E_I$ is a cycle on $A_0$ and its class $[E_I]$ is effective. If $C_I$ is zero, we set $E_I$ to zero as well. It follows from the definition of $d_i = d_{s_i}$ in the paragraph before \eqref{eq:masserwuestholz} that $d_i^2 \leq \deg \tilde{\psi}_iÊ\leq 4d_i^2$. Hence, we deduce that
\[Ê4^{\dim Y_i}d_i^{2\dim Y_i}[Z_i] \geq N\left(\tilde{\phi}_i\right)_{\ast}([Y_i]) \geq d_i^{2\dim Y_i}[Z_i] ,\]
where the class of
\[ Z_i = \sum_{I \subset \{1,\hdots,g'\}}{E_I} \]
is effective and we write $[U] \geq [V]$ if $[U]-[V]$ is an effective class.

It follows that
\begin{equation}\label{eq:intersectionnumberboundbelow}
\mathcal{M}_1^{\cdot \dim(\Phi(Y))} \cdot \Phi_{\ast}([Y])Ê\geq \frac{1}{N^m}\left(\prod_{i=1}^{m}{d_i^{2\dim Y_i}}\right)\mathcal{M}_1^{\cdot \dim(\Phi(Y))} \cdot [Z_1 \times \cdots \times Z_m].
\end{equation}

The right-hand side here is positive since
\[ \mathcal{M}_1^{\cdot \dim(\Phi(Y))} \cdot \Phi_{\ast}([Y])Ê\leq \frac{1}{N^m}\left(\prod_{i=1}^{m}{(4d_i^2)^{\dim Y_i}}\right)\mathcal{M}_1^{\cdot \dim(\Phi(Y))} \cdot [Z_1 \times \cdots \times Z_m]\]
and
\[ \mathcal{M}_1^{\cdot \dim(\Phi(Y))} \cdot \Phi_{\ast}([Y]) > 0.\]

This last inequality follows from the fact that the morphism $\Psi_1$ restricted to $\Phi(Y)$ is generically finite, which can be shown by adapting the proof of Lemme 2.1 in \cite{MR1765539}. We simply have to note that $\tilde{\phi}_i(Y_i)$ cannot have a positive-dimensional stabilizer since otherwise the same would hold for $Y_i$ and then $x_i$ would lie in a translate of a positive-dimensional abelian subvariety that is contained in $X_i$, which contradicts our assumption.

As a positive integer is greater than or equal to $1$, we obtain that
\[ \mathcal{M}_1^{\cdot \dim(\Phi(Y))} \cdot [Z_1 \times \cdots \times Z_m] \geq 1\]
and therefore 
\[ \mathcal{M}_1^{\cdot \dim(\Phi(Y))} \cdot \Phi_{\ast}([Y])Ê\geq \frac{1}{N^m}\left(\prod_{i=1}^{m}{d_i^{2\dim Y_i}}\right)\]
by \eqref{eq:intersectionnumberboundbelow}.

Since $a_i = 4(b_id_i)^2$ ($i=1,\hdots,m$), the lemma now follows from combining this inequality with \eqref{eq:pushpull} and \eqref{eq:intersectionfaltings}.
\end{proof}

\begin{lem}\label{lem:isogenycontrol}
Let $s \in S$ be such that $\mathcal{E}_s$ and $E_0$ are isogenous and let $\psi: \mathcal{E}_{s} \to E_0$ be an isogeny. Recall that we have embedded $\mathcal{E}_{s}$ and $E_0$ into $\mathbb{P}^5$. There exists a constant $\tilde{c}$, depending only on $A_0$, $\mathcal{A}$, and their (quasi-)projective immersions, such that $\psi$ is given by $6$ homogeneous polynomials of degree $\deg \psi$ in the coordinates on $\mathbb{P}^{5}$ and the height of the set of coefficients of all these polynomials is at most
\[ \tilde{c}(\deg \psi)^{10}\max\{h_{\overline{S}}(s),1\}.\]
\end{lem}

\begin{proof}
The embeddings into $\mathbb{P}^5$ yield line bundles $\mathcal{L}'_s$ on $\mathcal{E}_s$ and $L_0'$ on $E_0$. Recall that they are both sixth tensor powers of the line bundle associated to the divisor given by the zero element of the respective elliptic curve. It follows that $\psi^{\ast} L_0'$ corresponds to six times the divisor associated to the kernel of $\psi$. This divisor is linearly equivalent to $6(\deg \psi)0_{\mathcal{E}_s}$. It follows that $\psi^{\ast} L_0' \simeq \mathcal{L}_s^{'\otimes\deg\psi}$.

The coordinates on $\mathbb{P}^{5}$ pull back under $\psi$ to global sections of this line bundle and our first goal is to show that these sections can in fact be written as homogeneous polynomials of degree $\deg\psi$ in the coordinates on $\mathcal{E}_s \subset \mathbb{P}^{5}$. Alternatively, we aim to show that the divisor given by the pull-back of a coordinate hyperplane under $\psi$ is cut out by a single homogeneous form of degree $\deg \psi$. Since $H^1(\mathbb{P}^2,\mathcal{O}_{\mathbb{P}^2}(k)) = \{0\}$ for all integers $k$ by Theorem III.5.1(b) in \cite{MR0463157}, the embedding of $\mathcal{E}_s$ into $\mathbb{P}^2$ is projectively normal. Therefore, the same holds for the embedding of $\mathcal{E}_s$ into $\mathbb{P}^5$ and the desired claim follows.

Thus, the isogeny $\psi$ is given by six homogeneous polynomials $P_0,Ê\hdots, P_5$ of degree $d = \deg \psi$ in six variables. We know that $\psi$ maps the torsion points of $\mathcal{E}_s$ onto the torsion points of $E_0$. The equalities $\psi(p_i) = q_i$ ($i=1,2,\hdots$), where the $p_i = [p_{i,0} : \cdots : p_{i,5}]$ are the torsion points of $\mathcal{E}_s \subsetÊ\mathbb{P}^5$ in some order and the $q_i = [q_{i,0} : \cdots : q_{i,5}]$ are torsion points of $E_0 \subsetÊ\mathbb{P}^5$, give rise to homogeneous linear equations of the form
\[Êq_{i,k_1}P_{k_2}(p_{i,0},\hdots,p_{i,5}) - q_{i,k_2}P_{k_1}(p_{i,0},\hdots,p_{i,5}) = 0 \quad (0 \leq k_1 < k_2 \leq 5)\]
that the coefficients of the $P_k$ ($k=0,\hdots,5$) satisfy.

Let $D=6\binom{d+5}{5}$ be the number of coefficients of all the $P_k$ ($k=0,\hdots,5$) and let $\rho \leq D$ be the rank of the (infinite) system of linear equations. We may choose $\rho$ of these equations such that the resulting matrix has rank $\rho$. Consequently, there is a minor of dimension $\rho$ with non-vanishing determinant. Using Cramer's rule, we obtain a basis for the space of solutions by taking determinants of suitable $\rho \times \rho$-matrices.

The entries of such a matrix are either $0$ or have the form
\[ \pm q_{i,k}p_{i,0}^{j_0}\cdots p_{i,5}^{j_{5}},\]
where $j_0 + \cdots + j_5 = d$. If $\Delta$ is its determinant and $v$ is a normalized valuation of a number field $F$ containing all coefficients of the linear equations, we can estimate
\[Ê|\Delta|_v \leq \epsilon(v)\prod_{l=1}^{\rho}{\left(\max_{k}\{|q_{i_l,k}|_v\}\max_{n}\{|p_{i_l,n}|_v\}^d\right)},\]
where $\epsilon(v)$ equals $1$ or $\rho!$ according to whether $v$ is finite or infinite and $i_1,Ê\hdots, i_{\rho}$ are the indices occurring in the rows of the matrix. For a basis element $z = [\Delta_1 : \cdots : \Delta_D]$, where each $\Delta_k$ is either a determinant as above or zero, but not all $\Delta_k$ are zero, we obtain that
\[ h(z) \leq \log(\rho!)+\rho\max_{l}\{h_{E_0}(q_{i_l})\}+d\rho\max_{l}\{h^0_s(p_{i_l})\}.\]
Using $\rho \leq D$ and $\rho! \leq \rho^{\rho}$, we further deduce that
\[ h(z) \leq D\log D+D\max_{l}\{h_{E_0}(q_{i_l})\}+dD\max_{l}\{h^0_s(p_{i_l})\}.\]

Since $\widehat{h}_{E_0}(q_{i_l}) = \widehat{h}^{0}_{s}(p_{i_l}) = 0$, we obtain that $h_{E_0}(q_{i_l}) \leq c_{A_0}$ and $h^0_s(p_{i_l}) \leq \tilde{c}_2\max\{h_{\overline{S}}(s),1\}$ from \eqref{eq:heightcomparisonazero} and Lemma \ref{lem:heightcomparison} (for all $l$). Together with the above this implies that
\[ h(z) \leq \tilde{c}D^2\max\{1,h_{\overline{S}}(s)\}\]
for some constant $\tilde{c}$ that depends only on $A_0$, $\mathcal{A}$, and their (quasi-)projective immersions. Since $D \leq 36d^{5}$, the lemma follows as soon as we have shown that we can choose an element of the basis which indeed represents $\psi$.

Now by construction, each element of this basis corresponds to a sextuple of polynomials $(P_0, \hdots, P_{5})$ such that for each torsion point $p \in \mathcal{E}_s$ we have either $P_0(p) = \cdots = P_{5}(p) = 0$ or $[P_0(p) : \cdots : P_{5}(p)] = \psi(p)$. After discarding those basis elements that vanish everywhere, we can assume that each basis element represents $\psi$ on an open Zariski dense subset of $\mathcal{E}_s$. Here, we use that the torsion points lie Zariski dense in $\mathcal{E}_s$. We choose one such element $(P_0,\hdots,P_{5})$. Note that we have not discarded everything since we already know that there exists some sextuple of polynomials that represents $\psi$.

If $H_k$ is the scheme-theoretic intersection of $E_0$ with the coordinate hyperplane in $\mathbb{P}^5$ given by the vanishing of the $k$-th coordinate ($k=0,\hdots,5$), then $P_k$ certainly vanishes on $\psi^{-1}(H_k)$. Recall that the coordinate hyperplanes of $\mathbb{P}^5$ intersect $E_0$ transversely, so $H_k$ is reduced and $\psi^{-1}(H_k) \neq \mathcal{E}_s$. If this is the precise zero locus of each $P_k$, then we are done. Otherwise, the zero loci of all $P_k$ share an irreducible component of codimension $1$.

We identify $H_k$ with the divisor on $E_0$ that is associated to it. Each of the $P_k$ defines an effective divisor $D_k$ on $\mathcal{E}_s$ that is linearly equivalent to $\psi^{\ast}(H_k)$ (being a homogeneous polynomial of degree $d$ that does not vanish identically on $\mathcal{E}_s$ since $\psi^{-1}(H_k) \neq \mathcal{E}_s$). At the same time, the support of $D_k$ contains the support of $\psi^{\ast}(H_k)$. Since $H_k$ is reduced and $\psi$ is unramified, it follows that $\psi^{\ast}(H_k)$ is reduced and hence $D_k - \psi^{\ast}(H_k)$ is effective. If this difference does not vanish for some $k$, this would give us a non-trivial effective divisor that is linearly equivalent to $0$, a contradiction. Hence, the polynomials $(P_0,\hdots,P_5)$ define $\psi$ everywhere on $\mathcal{E}_s$.
\end{proof}

\begin{lem}\label{lem:isogenycontrolthree}
Let $s', s'' \in S$ and let $b', b'' \in \mathbb{N}$. Let $\tilde{\psi}_{s'}: \mathcal{E}_{s'}Ê\to E_0$ and $\tilde{\psi}_{s''}: \mathcal{E}_{s''} \to E_0$ be isogenies and set $\tilde{\phi}_{s'} = (\tilde{\psi}_{s'},\hdots,\tilde{\psi}_{s'},d' \cdotÊ\id_{E_1Ê\times \cdotsÊ\times E_{g'}}): \mathcal{A}_{s'} \to A_0$, $\tilde{\phi}_{s''} = (\tilde{\psi}_{s''},\hdots,\tilde{\psi}_{s''},d'' \cdotÊ\id_{E_1Ê\times \cdotsÊ\times E_{g'}}): \mathcal{A}_{s''} \to A_0$, where $d' = \left[\sqrt{\degÊ\tilde{\psi}_{s'}}\right]$ and $d'' = \left[\sqrt{\degÊ\tilde{\psi}_{s''}}\right]$. Recall that $\mathcal{A}_{s'}$, $\mathcal{A}_{s''}$, and $A_0$ are embedded into $\mathbb{P}^R$, inducing line bundles $\mathcal{L}_{s'}$ on $\mathcal{A}_{s'}$, $\mathcal{L}_{s''}$ on $\mathcal{A}_{s''}$, and $L_0$ on $A_0$. Let $\chi: \mathcal{A}_{s'} \times \mathcal{A}_{s''}Ê\to A_0$ be defined by $\chi(y',y'') = b'\tilde{\phi}_{s'}(y')+b''\tilde{\phi}_{s''}(y'')$ and let $\pr_1: \mathcal{A}_{s'} \timesÊ\mathcal{A}_{s''} \toÊ\mathcal{A}_{s'}$ and $\pr_2: \mathcal{A}_{s'} \timesÊ\mathcal{A}_{s''} \toÊ\mathcal{A}_{s''}$ be the canonical projections. Set $a' = 4(b'd')^2$ and $a'' = 4(b''d'')^2$. The following hold:

\begin{enumerate}[label=(\roman*)]
\item For each $y \in \mathcal{A}_{s'}Ê\times \mathcal{A}_{s''}$, there exists an injection from $\chi^{\ast} L_0$ into $\pr_1^{\ast}\mathcal{L}_{s'}^{\otimes 4a'} \otimes \pr_2^{\ast}\mathcal{L}_{s''}^{\otimes 4a''}$ that induces an isomorphism in an open neighbourhood of $y$. This injection can be chosen from a set of cardinality at most $(R+1)^5$.
\item The line bundle $\chi^{\ast}L_0$ is generated by a set of $R+1$ sections -- the pull-backs of the homogeneous coordinates on $A_0 \subset \mathbb{P}^R$ --, mapping under the chosen injection to bihomogeneous polynomials in the coordinates given by the embedding $\mathcal{A}_{s'} \times \mathcal{A}_{s''} \hookrightarrow \mathbb{P}^R \times \mathbb{P}^R$ of bidegree $(4a',4a'')$.
\item Furthermore, there exists a constant $\tilde{c}$, depending only on $A_0$, the family $\mathcal{A}$, and their (quasi-)projective immersions, such that the set of all coefficients of these $R+1$ polynomials has height at most $a'\delta'+a''\delta''$, where $\delta' \leq \tilde{c}d'^{18}\max\{h_{\overline{S}}(s'),1\}$, $\delta'' \leq \tilde{c}d''^{18}\max\{h_{\overline{S}}(s''),1\}$.
\end{enumerate}
\end{lem}

\begin{proof}
Define $\sigma_{b',b''}: A_0 \times A_0 \to A_0$ by $\sigma_{b',b''}(y',y'') = b'y'+b''y''$. We can show as in the proof of Lemma \ref{lem:isogenycontrol} that the embeddings of $E_0, E_1,Ê\hdots, E_{g'}$ into $\mathbb{P}^5$ are projectively normal. It follows by repeated application of the K\"unneth formula (Proposition 9.2.4 in \cite{MR1252397}) that the embedding of $A_0$ into $\mathbb{P}^R$ is projectively normal as well. We can therefore apply Proposition 5.2 from \cite{MR1765539}.

It yields a set of $(R+1)^3$ morphisms of invertible sheaves from $\sigma_{b',b''}^{\ast} L_0$ into $\pr_1^{\ast}L_0^{\otimes 4b'^2} \otimes \pr_2^{\ast}L_0^{\otimes 4b''^2}$ such that for each $z \in A_0 \times A_0$ one of them is an injection that restricts to an isomorphism in an open neighbourhood of $z$, where by abuse of notation $\pr_1$ and $\pr_2$ denote also the canonical projections $A_0 \times A_0 \to A_0$. We also obtain a set of $R+1$ sections -- the pull-backs of the homogeneous coordinates on $A_0 \subset \mathbb{P}^R$ -- that generate $\sigma_{b',b''}^{\ast} L_0$ such that by taking the union of the images of these sections under all injections, we obtain a set of at most $(R+1)^4$ bihomogeneous polynomials in the coordinates on $A_0 \times A_0$ of bidegree $(4b'^2,4b''^2)$. Given a choice of injection, the set of all coefficients of the resulting $R+1$ polynomials has height bounded by $\tilde{c}_{A_0}(b'^2+b''^2)$, where $\tilde{c}_{A_0}$ depends only on $A_0$ and its embedding into $\mathbb{P}^R$. We have $\chi = \sigma_{b',b''} \circ (\tilde{\phi}_{s'},\tilde{\phi}_{s''})$.

Let $j \in \{1,\hdots,g'\}$. It follows from the proof of Proposition 5.2 in \cite{MR1765539} (applied to $E_j \subset \mathbb{P}^5$) that there exist $6$ sextuples of homogeneous polynomials in the coordinates on $E_j$ of degree $4d'^2$ such that for each point of $E_j$ one of the sextuples describes the multiplication-by-$d'$ morphism in a Zariski open neighbourhood of that point and the set of all coefficients of any sextuple has height bounded by $c'd'^2$ for some constant $c'$ that depends only on the $E_i$ ($i=1,\hdots,g'$) and their embeddings into $\mathbb{P}^5$.

It follows from Lemma \ref{lem:isogenycontrol} that there exist $6$ sextuples of homogeneous polynomials in the coordinates on $\mathcal{E}_{s'}$ of degree $4d'^2$ such that for each point of $\mathcal{E}_{s'}$ one of the sextuples describes the isogeny $\tilde{\psi}_{s'}$ in a Zariski open neighbourhood of that point and the set of all coefficients of any sextuple has height bounded by $\tilde{c}(\deg \tilde{\psi}_{s'})^{10}\max\{h_{\overline{S}}(s'),1\}$.

By choosing on each elliptic factor one of these sextuples and multiplying together all possible combinations of their entries, we find a collection of $6^g = R+1$ $(R+1)$-tuples of homogeneous polynomials in the coordinates on $\mathcal{A}_{s'}$ of degree $4d'^2$ such that for each point of $\mathcal{A}_{s'}$ one of these $(R+1)$-tuples describes the isogeny $\tilde{\phi}_{s'}$ in a Zariski open neighbourhood of that point with respect to the given embeddings into $\mathbb{P}^R$.

Because of the shape of the Segre embedding, the height of the family of coefficients of each $(R+1)$-tuple of polynomials can be bounded from above by $g'c'd'^2+(g-g')\tilde{c}(\deg \tilde{\psi}_{s'})^{10}\max\{h_{\overline{S}}(s'),1\}$ and thus by $\tilde{c}d'^{20}\max\{h_{\overline{S}}(s'),1\}$ (after increasing $\tilde{c}$). We can do the same thing for $\tilde{\phi}_{s''}$ with $d''$ instead of $d'$ and $s''$ instead of $s'$. 

Now plugging in each of the $(R+1)^2$ combinations of these $(R+1)$-tuples of polynomials into the set of $(R+1)^4$ polynomials from the beginning of the proof gives us a set of $(R+1)^{6}$ bihomogeneous polynomials in the coordinates on $\mathcal{A}_{s'} \times \mathcal{A}_{s''}$ of bidegree $(4a',4a'')$, which are the images of a set of $R+1$ sections that generate $\chi^{\ast}L_0$ under $(R+1)^5$ different morphisms of invertible sheaves from $\chi^{\ast}L_0$ to $\pr_1^{\ast}\mathcal{L}_{s'}^{\otimes 4a'} \otimes \pr_2^{\ast}\mathcal{L}_{s''}^{\otimes 4a''}$. Each possibility for the morphism corresponds to one of the $(R+1)^2$ combinations from above together with one of the $(R+1)^3$ possibilities from the beginning of the proof. For each $y \in \mathcal{A}_{s'}Ê\times \mathcal{A}_{s''}$, one of these morphisms is an injection that restricts to an isomorphism in a neighbourhood of $y$. The sections are the pull-backs of the homogeneous coordinates on $A_0 \subset \mathbb{P}^R$.

We can bound the height of the family of coefficients of all $R+1$ polynomials corresponding to a choice of injection from above by
\begin{align*}
\tilde{c}_{A_0}(b'^2+b''^2)+\tilde{c}(4b'^2d'^{20}\max\{h_{\overline{S}}(s'),1\} +4b''^2d''^{20}\max\{h_{\overline{S}}(s''),1\}) \\
+4b'^2\log\binom{R+4d'^2}{R}+4b''^2\log\binom{R+4d''^2}{R}\\
+\log\binom{R+4b'^2}{R}+\log\binom{R+4b''^2}{R}.
\end{align*}
Here the last four summands are a very crude upper bound for the logarithm of the number of monomials that one obtains after multiplying out and before combining like terms and hence also an upper bound for the logarithm of the maximal number of equal monomials that are obtained in this way. The lemma now follows (after increasing $\tilde{c}$ again) from estimating
\[Ê\binom{R+4d'^2}{R} \leq (R+4d'^2)^R\]
and analogously for $d''$, $b'$, and $b''$.
\end{proof}

\begin{lem}\label{lem:isogenycontroltwo}
Let $s_1, \hdots, s_{m} \in S$ and let $b_1, \hdots, b_{m} \in \mathbb{N}$. Let $a_i$, $d_i$, and $\tilde{\phi}_{i}$ ($i=1,\hdots,m$) as well as $a = (a_1,\hdots,a_m)$ and $\mathcal{N}_a$ be defined as above. Let $\tilde{\Psi}: X_1 \timesÊ\cdots \times X_m \to A_0^{m-1}$ be the morphism given by $\tilde{\Psi}(y_1,\hdots,y_{m}) =$
\[(b_1\tilde{\phi}_{1}(y_1)+b_2\tilde{\phi}_{2}(y_2),\hdots,b_{m-1}\tilde{\phi}_{m-1}(y_{m-1})+b_{m}\tilde{\phi}_{m}(y_{m})).\]
Recall that $X_i \subset \mathcal{A}_{s_i}$ ($i=1,\hdots,m$) and $A_0$ are all embedded into $\mathbb{P}^{R}$. Let $q_1$, \dots, $q_{m-1}: A_0^{m-1} \to A_0$ be the canonical projections. The following hold:

\begin{enumerate}[label=(\roman*)]
\item For each $z \in X_1 \timesÊ\cdots \times X_m$, there exists an injection $\tilde{\Psi}^{\ast}(q_1^{\ast}L_0 \otimes \cdots \otimes q_{m-1}^{\ast}L_0) \hookrightarrow \mathcal{N}_a^{\otimes 8}$ that induces an isomorphism on an open neighbourhood of $z$. It can be chosen from a set of cardinality at most $(R+1)^{5m-3}$.
\item The line bundle $\tilde{\Psi}^{\ast}(q_1^{\ast}L_0 \otimes \cdots \otimes q_{m-1}^{\ast}L_0)$ is generated by $M=(R+1)^{m-1}$ sections -- the pull-backs of the homogeneous coordinates on $A_0^{m-1}Ê\subset (\mathbb{P}^R)^{m-1} \hookrightarrow \mathbb{P}^{M-1}$ --, mapping under the chosen injection to multihomogeneous polynomials of multidegree $8a = (8a_1,\hdots,8a_{m})$ in the multiprojective coordinates on $X_1 \times \cdots \times X_m \subset (\mathbb{P}^R)^m$.
\item Furthermore, there exists a constant $\tilde{c}$, depending only on $A_0$, the family $\mathcal{A}$, and their (quasi-)projective immersions, such that the set of all coefficients of these polynomials has height at most $\sum_{i=1}^{m}{a_iÊ\delta_i}$, where $\delta_i \leq \tilde{c}d_i^{18}\max\{h_{\overline{S}}(s_i),1\}$ ($i=1,\hdots,m$).
\end{enumerate}
\end{lem}

\begin{proof}
We apply Lemma \ref{lem:isogenycontrolthree} $m-1$ times with $(s',s'') = (s_i, s_{i+1})$ ($i=1,\hdots,m-1$). We obtain $m-1$ systems of at most $(R+1)^6$ bihomogeneous polynomials each. Multiplying together all possible combinations of one section from each system gives us at most $(R+1)^{6(m-1)}$ multihomogeneous polynomials of multidegree $(4a_1,8a_2,\hdots,8a_{m-1},4a_m)$, corresponding to the union of the images of a system of $M = (R+1)^{m-1}$ sections that generates $\tilde{\Psi}^{\ast}(q_1^{\ast}L_0 \otimes \cdots \otimes q_{m-1}^{\ast}L_0)$ under each possible combination of the injections furnished by Lemma \ref{lem:isogenycontrolthree} (some of them might not remain injections, when pulled back to $X_1Ê\timesÊ\cdots \times X_m$, and we discard those). We multiply each polynomial by each combination of a $(4a_1)$-th power of one of the $R+1$ coordinates on $X_1$ and a $(4a_m)$-th power of one of the $R+1$ coordinates on $X_m$. This yields at most $(R+1)^{6m-4}$ multihomogeneous polynomials of multidegree $8a$, still corresponding to the union of the images of a system of $M$ sections that generate $\tilde{\Psi}^{\ast}(q_1^{\ast}L_0 \otimes \cdots \otimes q_{m-1}^{\ast}L_0)$ under one of the thus obtained injections $\tilde{\Psi}^{\ast}(q_1^{\ast}L_0 \otimes \cdots \otimes q_{m-1}^{\ast}L_0) \hookrightarrow \mathcal{N}_a^{\otimes 8}$. The sections are the pull-backs of the homogeneous coordinates on $A_0^{m-1}Ê\subset (\mathbb{P}^R)^{m-1} \hookrightarrow \mathbb{P}^{M-1}$.

By Lemma \ref{lem:isogenycontrolthree}, we can bound the height of the family of coefficients of all multihomogeneous polynomials corresponding to one choice of injection by
\[ \sum_{i=1}^{m}{2\tilde{c}a_id_i^{18}\max\{h_{\overline{S}}(s_i),1\}}+\sum_{i=1}^{m}{2\log\binom{R+4a_i}{R}},\]
where the second summand is again a crude upper bound for the logarithm of the number of monomials obtained after multiplying out and before combining like terms and hence also an upper bound for the logarithm of the maximal number of equal monomials that are obtained in this way. The lemma follows from
\[\binom{R+4a_i}{R} \leq (R+4a_i)^R \quad (i=1,\hdots,m)\]
after increasing $\tilde{c}$ again.
\end{proof}

\begin{proof} (of Theorem \ref{thm:vojtaheightbound})
Putting together everything we did so far in this section, one can see that we have now proven Theorem \ref{thm:vojtaheightbound}. We summarize the proof: We divide $\Gamma \otimes \mathbb{R}$ into finitely many cones such that for each choice of $\zeta_i$, $\zeta_{i+1}$ in one of these cones the inequality \eqref{eq:littleangle} is satisfied. This is possible since $\Gamma \otimes \mathbb{R}$ is a finite-dimensional vector space with respect to $\lVert \cdotÊ\rVert$ and $c_1$ satisfies $c_1 \preceq 1$ by our choice of parameters and \eqref{eq:celsiusorfahrenheit}. If we prove a height bound of the desired form for each cone, we also get a global one by taking the maximum over the (finitely many) implicit constants.

If a cone contains only finitely many of the points $p$ whose height we would like to bound, the desired bound holds trivially. If a cone contains infinitely many such points, we can suppose that we can find among them points $x_1$, \dots, $x_m$ which satisfy \eqref{eq:growingfast}, \eqref{eq:heightboundpointone}, and \eqref{eq:heightboundpoint}; otherwise, a bound of the desired form follows from \eqref{eq:masserwuestholz}, \eqref{eq:sandwich}, Lemma \ref{lem:heightbound}, Lemma \ref{lem:heightcomparison}, and the choice of parameters in the generalized Vojta-R\'emond inequality. But then, by \eqref{eq:fastgrowth}, \eqref{eq:largeheight}, Lemma \ref{lem:voidbringer}, and Lemma \ref{lem:isogenycontroltwo}, the generalized Vojta-R\'emond inequality in the form of Theorem \ref{thm:vojta} can be applied to yield a contradiction with \eqref{eq:superbadinequality}. So Theorem \ref{thm:vojtaheightbound} follows.
\end{proof}

\section{Application of the Pila-Zannier strategy}\label{sec:pila-zannier}

In this section, we will apply the Pila-Zannier strategy to deduce the following Proposition \ref{prop:pila-zannier} from Theorem \ref{thm:vojtaheightbound}. This part is very similar to the case of curves that was treated in \cite{D18} and no substantial new difficulties appear, so we often refer to that article and try to keep the exposition short. In particular, we will use terminology from the theory of o-minimal structures without further explanation and refer the reader to Section 5 of \cite{D18} for definitions. In order to speak of definable subsets of powers of $\mathbb{C}$, we identify $\mathbb{C}$ with $\mathbb{R}^2$ through use of the maps $\Re$ and $\Im$.

\begin{prop}\label{prop:pila-zannier}
Let $\mathcal{A}Ê\to S$ and $A_0$ be as in Section \ref{sec:heightbounds}. Let $\mathcal{V} \subsetÊ\mathcal{A}$ be an irreducible subvariety that dominates $S$. Let $\mathcal{U}$ be the set of points $p \in \mathcal{V} \cap \mathcal{A}_\Gamma$ that do not lie in a translate of a positive-dimensional abelian subvariety of $\mathcal{A}_{\pi(p)}$ contained in $\mathcal{V}_{\pi(p)}$. Then there exists a subgroup $\Gamma' \subset \mathcal{A}_\xi^{\overline{\bar{\mathbb{Q}}(S)}/\bar{\mathbb{Q}}}(\bar{\mathbb{Q}})$ of finite rank such that for all but finitely many $p \in \mathcal{U}$ there exists an irreducible curve $\mathcal{C}$ such that $p \in \mathcal{C}$, $\pi(\mathcal{C}) = S$, $\mathcal{C} \subset \mathcal{V}$, and $\mathcal{C}_\xi \subset (\mathcal{A}_\xi)_{\tors} + \Tr\left(\Gamma'\right)$.
\end{prop}

\begin{proof}
As in Section \ref{sec:heightbounds}, we can assume without loss of generality that $\Gamma$ is invariant under $\End(A_0)$. For each $s \in S$ such that $A_0$ and $\mathcal{A}_s$ are isogenous, we fix an isogeny $\phi_s: A_0 \to \mathcal{A}_s$ as in Section \ref{sec:heightbounds}. Let $p$ be a point in $\mathcal{U}$. By Lemma \ref{lem:technicallemma}, we have $p = \phi_{\pi(p)}(\gamma)$ for some $\gamma \in \Gamma$. By Theorem \ref{thm:vojtaheightbound}, we have $h_{\pi(p)}(p)Ê\preceq [K(p):K]$, where $K$ is as in Section \ref{sec:heightbounds}. It follows from Lemma \ref{lem:heightbound}, Lemma \ref{lem:heightcomparison}, \eqref{eq:masserwuestholz}, and \eqref{eq:sandwich} that also $\widehat{h}_{A_0}(\gamma) \preceq [K(p):K]$. If $N$ is the smallest natural number such that $N\gamma = a_1\gamma_1 + \cdots +a_r\gamma_r \in \bigoplus_{i=1}^{r}\mathbb{Z}\gamma_i$, then we can show as in Proposition 4.3 in \cite{D18} that $\max\{N,|a_1|,\hdots,|a_r|\} \preceq [K(p):K]$. If $\sigma \in \Gal(\bar{\mathbb{Q}}/K)$, then we deduce that $\sigma(p) = \sigma\left(\phi_{\pi(p)}\right)(\sigma(\gamma))$, where $\sigma$ acts on algebraic points and maps in the usual way and $\sigma\left(\phi_{\pi(p)}\right)$ is an isogeny between $A_0$ and $\mathcal{A}_{\pi(\sigma(p))}$. It follows that there is $\gamma_\sigma \in \Gamma$ with $\sigma(p) = \phi_{\pi(\sigma(p))}(\gamma_\sigma)$. Since $h_{\pi(p)}(p) = h_{\pi(\sigma(p))}(\sigma(p))$ and $[K(p):K] = [K(\sigma(p)):K]$, we find that also
\begin{equation}\label{eq:cleveland}
\max\{N_\sigma,|a_{\sigma,1}|,\hdots,|a_{\sigma,r}|\} \preceq [K(p):K]
\end{equation}
for the analogous quantities for $\gamma_\sigma$.

There is a uniformization map $e: \mathbb{H}Ê\times \mathbb{C} \to \mathcal{E}(\mathbb{C})$. Its restriction to
\[ \{(\tau,z) \in \mathcal{F}Ê\times \mathbb{C}; z = x+y\tau, x,y \in [0,1)\}\]
is surjective and definable in the o-minimal structure $\mathbb{R}_{\an,\exp}$, where $\mathcal{F}$ is a fundamental domain in $\mathbb{H}$ for the action of a certain subgroup of $\SL_2(\mathbb{Z})$ of finite index. For details, see for example Sections 2 and 3 of \cite{MR2520786}; the definability follows from \cite{MR2134454}.

We choose $\tau_0,\tau_1,\hdots,\tau_{g'} \in \mathcal{F}$ such that $E_i(\mathbb{C})$ is isomorphic to $\mathbb{C}/(\mathbb{Z}+\tau_i\mathbb{Z})$. There are uniformization maps $e_i: \mathbb{C}Ê\to E_i(\mathbb{C})$ with kernel $\mathbb{Z}+\tau_i\mathbb{Z}$ that are definable in the o-minimal structure $\mathbb{R}_{\an}$ (and hence in $\mathbb{R}_{\an,\exp}$) when restricted to $\{x+y\tau_i; x,y \in [0,1)\}$ ($i=0,\hdots,g'$).

We can define a uniformization map $\exp: \mathbb{H} \times \mathbb{C}^g \to \mathcal{A}(\mathbb{C}) = (\mathcal{E} \times_{Y(2)} \cdots \times_{Y(2)} \mathcal{E})(\mathbb{C}) \times E_1(\mathbb{C}) \times \cdotsÊ\times E_{g'}(\mathbb{C})$ by
\[\exp(\tau,z_1,\hdots,z_g) = (e(\tau,z_1),\hdots,e(\tau,z_{g-g'}),e_1(z_{g-g'+1}),\hdots,e_{g'}(z_g)).\]
It has a fundamental domain
\begin{align*}
U = \{ (\tau,z_1,\hdots,z_g) \in \mathcal{F} \times \mathbb{C}^g; z_i = x_i+y_i\tau, z_{g-g'+j} = x_{g-g'+j}+y_{g-g'+j}\tau_j,Ê\\
 i=1,\hdots,g-g', j=1,\hdots,g', x_1,\hdots,x_g,y_1,\hdots,y_gÊ\inÊ[0,1)\},
\end{align*}
restricted to which it is surjective and definable in $\mathbb{R}_{\an,\exp}$. In order to speak of definability, we have to fix an embedding of $\mathcal{A}(\mathbb{C})$ into projective space and we can take the one from Section \ref{sec:heightbounds}. 

We set
\[ P = \begin{pmatrix} \tau_1 & & \\ &Ê\ddots & \\ & & \tau_{g'}\end{pmatrix}.\]
For $\tau \in \mathbb{H}$, we set 
\[Ê\Pi_{\tau} = \begin{pmatrix} \tau I_{g-g'} &Ê\\ & P\end{pmatrix}\]
and
\[Ê\Omega_{\tau} = \begin{pmatrix} \Pi_{\tau} & I_g\end{pmatrix}.\]
We can find $(\tau_\sigma,x_\sigma) \in \mathbb{H} \timesÊ[0,1)^{2g}$ such that $(\tau_\sigma,\Omega_{\tau_\sigma}x_\sigma) \in U$ and $\exp(\tau_\sigma,\Omega_{\tau_\sigma}x_\sigma) = \sigma(p)$.

We can also define a uniformization map $\exp_0: \mathbb{C}^g \to A_0(\mathbb{C}) = E_0(\mathbb{C})^{g-g'} \times E_1(\mathbb{C}) \timesÊ\cdotsÊ\times E_{g'}(\mathbb{C})$ by
\[\exp_0(z_1,\hdots,z_g) = (e_0(z_1),\hdots,e_0(z_{g-g'}),e_1(z_{g-g'+1}),\hdots,e_{g'}(z_g)),\]
where we consider $A_0(\mathbb{C})$ as embedded into projective space as in Section \ref{sec:heightbounds}. The uniformization map is then surjective and definable in $\mathbb{R}_{\an}$ when restricted to a fundamental domain
\begin{align*}
\{(x_1+y_1\tau_0,\hdots,x_{g-g'}+y_{g-g'}\tau_0,x_{g-g'+1}+y_{g-g'+1}\tau_1,\hdots,x_g+y_g\tau_{g'}); \\
 x_1,\hdots,x_g,y_1,\hdots,y_gÊ\in [0,1)\}.
\end{align*}
The points $\gamma_1,\hdots,\gamma_r$ have pre-images $\tilde{\gamma}_1,\hdots,\tilde{\gamma}_r$ in this fundamental domain under $\exp_0$.

The isogeny $\phi_{\pi(\sigma(p))}$ lifts under the uniformizations $\exp_0$ and $\exp(\tau_\sigma,\cdot)$ to a linear map from $\mathbb{C}^g$ to $\mathbb{C}^g$ given by
\[Ê\tilde{\alpha}_{\sigma} = \begin{pmatrix}Ê\alpha_\sigma I_{g-g'} & \\ & d_{\pi(\sigma(p))} I_{g'} \end{pmatrix},\]
where $\phi_{\pi(\sigma(p))} = (\psi_{\pi(\sigma(p))},\hdots,\psi_{\pi(\sigma(p))},d_{\pi(\sigma(p))}Ê\cdotÊ\id_{E_1Ê\times \cdots \times E_{g'}})$ and $\alpha_{\sigma}Ê\in \mathbb{C}$ is the analytic representation of $\psi_{\pi(\sigma(p))}$ with respect to the given uniformizations of $E_0(\mathbb{C})$ and $\mathcal{E}_{\pi(\sigma(p))}(\mathbb{C})$. There exists a matrix $\Psi_\sigma = \begin{pmatrix} b_{\sigma,1} b_{\sigma,2}\\ b_{\sigma,3} b_{\sigma,4}\end{pmatrix} \in \M_{2}(\mathbb{Z}) \cap \GL_2(\mathbb{Q})$ of determinant $\deg \psi_{\pi(\sigma(p))}$ such that
\begin{equation}\label{eq:jasnah}
\alpha_{\sigma}Ê\begin{pmatrix} \tau_0 \quad 1 \end{pmatrix} =Ê\begin{pmatrix} \tau_{\sigma} \quad 1 \end{pmatrix}Ê\Psi_{\sigma}.
\end{equation} It follows from \eqref{eq:masserwuestholz} and the definition of $d_{\pi(\sigma(p))}$ that
\begin{equation}\label{eq:gilliam}
d_{\pi(\sigma(p))} \leq \sqrt{\deg \psi_{\pi(\sigma(p))}} \preceq [K(p):K].
\end{equation}

Since $\exp(\tau_\sigma,\Omega_{\tau_\sigma}x_\sigma) = \sigma(p) = \phi_{\pi(\sigma(p))}(\gamma_\sigma)$, we deduce that $\exp_0(\tilde{\alpha}_{\sigma}^{-1}\Omega_{\tau_\sigma}x_{\sigma}) \in \gamma_\sigma + \kerÊ\phi_{\pi(\sigma(p))}$. As $\kerÊ\phi_{\pi(\sigma(p))}$ is annihilated by $\left(\det \Psi_{\sigma}\right)d_{\pi(\sigma(p))}$, we deduce that $\exp_0(\left(\det \Psi_{\sigma}\right)d_{\pi(\sigma(p))}\tilde{\alpha}_{\sigma}^{-1}\Omega_{\tau_\sigma}x_{\sigma}) = \left(\det \Psi_{\sigma}\right)d_{\pi(\sigma(p))}\gamma_\sigma$. It follows that
\[\left(\det \Psi_{\sigma}\right)d_{\pi(\sigma(p))}\left(N_\sigma\tilde{\alpha}_{\sigma}^{-1}\Omega_{\tau_\sigma}x_{\sigma}-a_{\sigma,1}\tilde{\gamma}_1-\cdots-a_{\sigma,r}\tilde{\gamma}_r\right) \in \ker \exp_0,\]
so there exists $R_\sigma \in \mathbb{Z}^{2g}$ such that
\begin{equation}\label{eq:jones}
\left(\det \Psi_{\sigma}\right)d_{\pi(\sigma(p))}\left(N_\sigma\tilde{\alpha}_{\sigma}^{-1}\Omega_{\tau_\sigma}x_{\sigma}-a_{\sigma,1}\tilde{\gamma}_1-\cdots-a_{\sigma,r}\tilde{\gamma}_r\right) = \Omega_{\tau_0}R_\sigma.
\end{equation}

It follows from \eqref{eq:jasnah} that
\[Ê\tau_0 = \frac{\tau_{\sigma}b_{\sigma,1}+b_{\sigma,3}}{\tau_{\sigma}b_{\sigma,2}+b_{\sigma,4}} \]
and therefore
\begin{equation}\label{eq:palin}
\tau_{\sigma} = \frac{\tau_0 b_{\sigma,4}-b_{\sigma,3}}{-\tau_0b_{\sigma,2}+b_{\sigma,1}}.
\end{equation}

Since $\tau_\sigma \in \mathcal{F}$, Theorem 1.1 in \cite{O16} with ${\bf G} = \GL_2$, $n=2$, and $\rho = \id_{\GL_2}$ shows that $\lVertÊ\Psi_{\sigma}Ê\rVert \preceq \detÊ\Psi_\sigma = \degÊ\psi_{\pi(\sigma(p))}$ and hence
\begin{equation}\label{eq:chapman}
\lVertÊ\Psi_\sigma \rVert \preceq [K(p):K]
\end{equation}
by \eqref{eq:masserwuestholz}.

For $D \in \mathbb{N}$ and $B = \begin{pmatrix} B_1 & B_2 \\ B_3 & B_4Ê\end{pmatrix} \in \M_2(\mathbb{Z})$, we define
\[ \beta_i(B) = \begin{pmatrix} B_i I_{g-g'} & 0Ê\\ 0 & 0\end{pmatrix} \in \M_g(\mathbb{Z}) \quad (i=1,\hdots,4),Ê\]
\[ \beta_5(D) =\begin{pmatrix} 0 & 0Ê\\ 0 & DI_{g'}\end{pmatrix} \in \M_g(\mathbb{Z})\]
and
\[\beta(B,D) = \begin{pmatrix} \beta_1(B)+\beta_5(D) & \beta_2(B) \\ \beta_3(B) & \beta_4(B)+\beta_5(D)\end{pmatrix} \in \M_{2g}(\mathbb{Z}).\]

We can deduce from \eqref{eq:jasnah} that
\[Ê\tilde{\alpha}_{\sigma}\Omega_{\tau_0} = \Omega_{\tau_\sigma}\beta\left(\Psi_\sigma,d_{\pi(\sigma(p))}\right).\]
Here, $\beta\left(\Psi_\sigma,d_{\pi(\sigma(p))}\right)$ is invertible and it follows that
\begin{equation}\label{eq:cleese}
\tilde{\alpha}_{\sigma}^{-1}\Omega_{\tau_\sigma} = \Omega_{\tau_0}\beta\left(\Psi_\sigma,d_{\pi(\sigma(p))}\right)^{-1}.
\end{equation}
Together with \eqref{eq:jones}, this implies that
\begin{equation}\label{eq:smith}
\left(\det \Psi_{\sigma}\right)d_{\pi(\sigma(p))}\left(N_\sigma\Omega_{\tau_0}\beta\left(\Psi_\sigma,d_{\pi(\sigma(p))}\right)^{-1}x_{\sigma}-a_{\sigma,1}\tilde{\gamma}_1-\cdots-a_{\sigma,r}\tilde{\gamma}_r\right) = \Omega_{\tau_0}R_\sigma.
\end{equation}

It now follows from \eqref{eq:cleveland}, \eqref{eq:gilliam}, and \eqref{eq:chapman} as well as $x_{\sigma} \in [0,1)^{2g}$ that
\begin{equation}\label{eq:gavilar}
\lVert R_\sigmaÊ\rVert \preceq [K(p):K]
\end{equation}
since we can solve \eqref{eq:smith} for $R_\sigma$ by conjugating and obtaining $\begin{pmatrix} \Omega_{\tau_0} \\Ê\overline{\Omega_{\tau_0}}Ê\end{pmatrix} R_\sigma$ on the right-hand side, where $\begin{pmatrix} \Omega_{\tau_0} \\Ê\overline{\Omega_{\tau_0}}Ê\end{pmatrix}$ is invertible.

We set $X = \exp|_{U}^{-1}(\mathcal{V}(\mathbb{C}))$. From now on, ``definable" will always mean ``definable in the o-minimal structure $\mathbb{R}_{\an,\exp}$". Consider the definable set
\begin{align*}
Z=\{(A_1,\hdots,A_r,M,R,B_1,B_2,B_3,B_4,D,\tau,x) \in \mathbb{R}^{r+1+2g+5} \times \mathbb{H}Ê\timesÊ\mathbb{R}^{2g}; \\
(\tau,\Omega_\tau x) \in X, B=\begin{pmatrix}B_1 & B_2 \\ B_3 & B_4 \end{pmatrix}, (\det B)D \neq 0, \tau(-B_2\tau_0+B_1)= B_4 \tau_0-B_3, \\
M > 0, \Omega_{\tau_0}R= (\det B)D\left(M\Omega_{\tau_0}\beta(B,D)^{-1}x-A_1\tilde{\gamma}_1-\cdots-A_r\tilde{\gamma}_r\right)\}.
\end{align*}

Let $\pi_1: Z \to \mathbb{R}^{r+1+2g+5}$ and $\pi_2: Z \to \mathbb{H} \times \mathbb{R}^{2g}$ be the canonical projections and let $\Sigma$ be the set of points
\[(a_{\sigma,1},\hdots,a_{\sigma,r},N_{\sigma},R_{\sigma},b_{\sigma,1},\hdots,b_{\sigma,4},d_{\pi(\sigma(p))},\tau_{\sigma},x_{\sigma}) \quad (\sigma \in \Gal(\bar{\mathbb{Q}}/K)).\]
It follows from \eqref{eq:palin} and \eqref{eq:smith} that $\Sigma \subset Z$. Furthermore, we have $|\pi_2(\Sigma)| = [K(p):K]$ and it follows from \eqref{eq:cleveland}, \eqref{eq:gilliam}, \eqref{eq:chapman}, and \eqref{eq:gavilar} that
\[Ê\max\{|a_{\sigma,1}|,\hdots,|a_{\sigma,r}|,N_{\sigma},\lVert R_{\sigma} \rVert, |b_{\sigma,1}|,\hdots,|b_{\sigma,4}|,d_{\pi(\sigma(p))}\}Ê\preceq [K(p):K].\]

If $[K(p):K]$ is sufficiently big (which by Lemma \ref{lem:heightbound}, Theorem \ref{thm:vojtaheightbound}, and Northcott's theorem excludes only finitely many $p \in \mathcal{U}$), we can apply Corollary 7.2 from \cite{MR3552014} and find that there exists a continuous function $\alpha: [0,1] \to Z$ such that $\pi_1Ê\circ \alpha$ is semialgebraic, $\pi_2 \circ \alpha$ is non-constant, $\pi_2(\alpha(0)) \in \pi_2(\Sigma)$, and $\alpha|_{(0,1)}$ is real analytic; note that $\mathbb{R}_{\an,\exp}$ admits analytic cell decomposition by \cite{MR1264338}. In the following, we use the variables $\tau, B, \hdots$ to denote the corresponding coordinate functions on $Z$. It follows from $\tau = \frac{B_4\tau_0-B_3}{-B_2\tau_0+B_1}$ and
\[\Omega_{\tau_0}\beta(B,D)^{-1}x = M^{-1}\left(A_1\tilde{\gamma}_1+\cdots+A_r\tilde{\gamma}_r+(\det B)^{-1}D^{-1}\Omega_{\tau_0}R\right)\]
that $\alpha$ itself is a semialgebraic map. Note that we can solve the last equation for $x$ by conjugating and obtaining the invertible matrix $\begin{pmatrix} \Omega_{\tau_0}Ê\\ \overline{\Omega_{\tau_0}}Ê\end{pmatrix}$. Let $\psi: \mathbb{H} \times \mathbb{R}^{2g}Ê\to \mathbb{H} \times \mathbb{C}^g$ be defined by $\psi(\tau,x) = (\tau,\Omega_\tau x)$. It follows that $\delta = \psi \circ \pi_2 \circ \alpha: [0,1] \to X$ is a non-constant continuous semialgebraic map and $\exp(\delta(0))$ is equal to some Galois conjugate of $p$.

We can now deduce from Theorem 5.1 in \cite{MR3164515} and Lemma 4.1 in \cite{MR3020307} that some Galois conjugate of $p$ and hence also $p$ itself lies in a positive-dimensional weakly special subvariety of the product of elliptic modular surfaces and elliptic curves $\mathcal{E}(\mathbb{C})^{g-g'} \times E_1(\mathbb{C}) \times \cdots \times E_{g'}(\mathbb{C})$ that is contained in $\mathcal{V}(\mathbb{C})$. In order to apply Theorem 5.1 from \cite{MR3164515}, we view the factors $E_i(\mathbb{C})$ ($i=1,\hdots,g'$) as fibers of some elliptic modular surfaces as defined in Section 2.2 of \cite{MR3164515}. For the definition of a weakly special subvariety, we refer to Section 4 of \cite{MR3164515}. The analogue of Theorem 5.1 in \cite{MR3164515} for arbitrary connected mixed Shimura varieties has been proven later by Gao in \cite{G152}.

At the same time, the point $p$ cannot lie in a translate of a positive-dimensional abelian subvariety of $\mathcal{A}_{\pi(p)}(\mathbb{C})$ contained in $\mathcal{V}_{\pi(p)}(\mathbb{C})$. It follows that $p$ must belong to an irreducible curve $\mathcal{C}$, a priori defined over $\mathbb{C}$, such that $\emptyset \neq \mathcal{C}_\xi \subset (\mathcal{A}_\xi)_{\tors} + \Tr\left(\mathcal{A}_\xi^{\overline{\bar{\mathbb{Q}}(S)}/\bar{\mathbb{Q}}}(\mathbb{C})\right)$ (we base change freely between $\bar{\mathbb{Q}}$ and $\mathbb{C}$ and between $\overline{\bar{\mathbb{Q}}(S)}$ and an algebraic closure of the function field of the base change of $S$ to $\mathbb{C}$). As $p \in \mathcal{A}_{\Gamma}$, we can first replace $\mathcal{A}_\xi^{\overline{\bar{\mathbb{Q}}(S)}/\bar{\mathbb{Q}}}(\mathbb{C})$ by $\mathcal{A}_\xi^{\overline{\bar{\mathbb{Q}}(S)}/\bar{\mathbb{Q}}}(\bar{\mathbb{Q}})$ -- in particular, $\mathcal{C}$ is defined over $\bar{\mathbb{Q}}$~-- and then $\mathcal{A}_\xi^{\overline{\bar{\mathbb{Q}}(S)}/\bar{\mathbb{Q}}}(\bar{\mathbb{Q}})$ by a subgroup $\Gamma'$ of finite rank.
\end{proof}

\section{Proof of Theorem \ref{thm:main}}\label{sec:mainproof}

We can assume without loss of generality that $A_0 = E_0^{g-g'}Ê\times E_1Ê\times \cdots \times E_{g'}$ for elliptic curves as in the hypothesis of Theorem \ref{thm:main}. We can also assume that $\pi(\mathcal{V}) = S$ since otherwise $\mathcal{V}$ is contained in $\mathcal{A}_s$ for some $s \in S$ and Theorem \ref{thm:main} then becomes an instance of the Mordell-Lang conjecture, proven by Vojta, Faltings, and Hindry.

It then follows that infinitely many fibers of $\mathcal{A}$ are pairwise isogenous. We deduce from Theorem 3 in \cite{MR2881307} by Habegger-Pila that the generic fiber of $\mathcal{A}$ must be isogenous (over $\overline{\bar{\mathbb{Q}}(S)}$) to $E^{g-g''} \times E_1' \times \cdots \times E_{g''}'$, where $E$ is an elliptic curve defined over $\overline{\bar{\mathbb{Q}}(S)}$ and $E_1', \hdots, E_{g''}'$ are elliptic curves defined over $\bar{\mathbb{Q}}$. The fact that infinitely many fibers of $\mathcal{A}$ are isogenous to $A_0$ and that $\Hom\left(\mathcal{A}_\xi^{\overline{\bar{\mathbb{Q}}(S)}/\bar{\mathbb{Q}}},E_0\right) = \{0\}$ and therefore $ \Hom(E_1'\times \cdots \times E_{g''}',E_0) = \{0\}$ as well as $\Hom(E_i,E_0) = \{0\}$ ($i=1,\hdots,g'$) implies that $g' = g''$. In particular, we have $g' < g$ since $\mathcal{A}$ is not isotrivial. We can deduce furthermore that after a permutation of $E_1, \hdots, E_{g'}$ $E_i'$ is isogenous to $E_i$ ($i=1,\hdots,g'$). We can then assume without loss of generality that $E_i' = E_i$ ($i=1,\hdots,g'$).

Consider the set of isomorphism classes of pairs $(\mathcal{A} \to S,A_0)$ such that
\begin{enumerate}
\item $S$ is a smooth and irreducible curve,
\item $\mathcal{A}$ is not isotrivial,
\item $\mathcal{A}_{\xi}$ is isogenous (over $\overline{\bar{\mathbb{Q}}(S)}$) to $E^{g-g'}Ê\times E_1 \times \cdots \times E_{g'}$ for $g' \geq 0$, an elliptic curve $E$ defined over $\overline{\bar{\mathbb{Q}}(S)}$, and elliptic curves $E_1,\hdots,E_{g'}$ defined over $\bar{\mathbb{Q}}$ that satisfy $\Hom(E_i,E_j) = \{0\}$ ($i \neq j$), and
\item $A_0$ is isogenous to $E_0^{g-g'} \times E_1Ê\times \cdots \times E_{g'}$, where $E_0$ is an elliptic curve defined over $\bar{\mathbb{Q}}$, $\Hom(E_0,E_i) = \{0\}$ ($i=1,\hdots,g'$), and either $g-g' = 1$ or $\End(E_0) = \mathbb{Z}$.
\end{enumerate}

This set is stable in the sense of Definition \ref{defn:stable}. By Proposition \ref{prop:reductionstep}, we can therefore assume that the union of all translates of positive-dimensional abelian subvarieties of $\mathcal{A}_s$ that are contained in $\mathcal{V}_s$ for some $s \in S$ is not Zariski dense in $\mathcal{V}$.

After replacing $S$ by an open subset of a finite cover, $\mathcal{A}$ by its corresponding pull-back, and $\mathcal{V}$ by an irreducible component of its pull-back, we can assume that there exists an elliptic scheme $\mathcal{E}'$ over $S$ and a surjective homomorphism $\alpha: \mathcal{A} \toÊ\mathcal{E}' \times_S \cdotsÊ\times_S \mathcal{E}'  \times_{S} (E_1 \times \cdots \times E_{g'} \times S)$ of abelian schemes over $S$ that restricts to an isogeny on each fiber. For this, we use our hypothesis on the generic fiber of $\mathcal{A}$ together with Proposition 8 from Section 1.2 and Theorem 3 from Section 1.4 of \cite{MR1045822}.

Arguing along the lines of the proof of Lemma 5.4 in \cite{MR3181568}, we can then assume that $\mathcal{A} = (\mathcal{E}Ê\times_S \cdots \times_{S} \mathcal{E}) \times_{S} (E_1 \times \cdots \times E_{g'} \times S)$, where $S = Y(2)$, $\mathcal{E}$ is the Legendre family, and $E_1,Ê\hdots, E_{g'}$ are some fixed elliptic curves over $\bar{\mathbb{Q}}$. Thus, we are in the situation of Sections \ref{sec:heightbounds} and \ref{sec:pila-zannier}. Since the union of all translates of positive-dimensional abelian subvarieties of $\mathcal{A}_s$ that are contained in $\mathcal{V}_s$ for some $s \in S$ is not Zariski dense in $\mathcal{V}$ and a finite set of points is not Zariski dense in $\mathcal{V}$, it follows from Proposition \ref{prop:pila-zannier} that there exists a subgroup $\Gamma' \subset \mathcal{A}_\xi^{\overline{\bar{\mathbb{Q}}(S)}/\bar{\mathbb{Q}}}(\bar{\mathbb{Q}})$ of finite rank such that the irreducible curves $\mathcal{C}$ that are contained in $\mathcal{V}$ and satisfy $\emptyset \neq \mathcal{C}_\xi \subset (\mathcal{A}_\xi)_{\tors} + \Tr\left(\Gamma'\right)$ lie Zariski dense in $\mathcal{V}$. We can then apply Lemma \ref{lem:weaklyspecialdense} to finish the proof.
\qed

\section*{Acknowledgements}
I thank my advisor Philipp Habegger for suggesting to study this problem and to apply a generalized Vojta-R\'emond inequality in this context as well as for his constant support. I thank Philipp Habegger and Ga\"{e}l R\'{e}mond for many useful and interesting discussions and for helpful comments on a preliminary version of this article. I thank the Institut Fourier in Grenoble for its hospitality, which I enjoyed while writing substantial portions of this article. This work was supported by the Swiss National Science Foundation as part of the project ``Diophantine Problems, o-Minimality, and Heights", no. 200021\_165525.
\appendix
\section{Generalized Vojta-R\'emond inequality}\label{sec:genvojtaineq}
Let $m \geq 2$ be an integer and let $X_1, \hdots, X_m$ be a family of irreducible positive-dimensional projective varieties, defined over $\bar{\mathbb{Q}}$. In the present work, we use a generalization \cite{D182} of R\'emond's results of \cite{MR2233693} to the case of an algebraic point $x = (x_1,\hdots,x_m)$ in the product $X = X_1 \times \cdots \times X_m$. Let us briefly recall the hypotheses which come into play. We follow the exposition of \cite{D182}.

For an $m$-tuple $a = (a_1,\hdots,a_m)$ of positive integers, we write
\[Ê\mathcal{N}_a = \bigotimes_{i=1}^{m}{p_i^{\ast}\mathcal{L}_i^{\otimes a_i}},\]
where $\mathcal{L}_i$ is a fixed very ample line bundle on $X_i$ and $p_i: X_1 \times \cdots \times X_m \to X_i$ is the natural projection. We relate $a$ to a nef line bundle $\mathcal{M}$ on $X$ which satisfies some further conditions, specified below.

By (a system of) homogeneous coordinates for a very ample line bundle we mean the set of pull-backs of the homogeneous coordinates on some $\mathbb{P}^{N'}$ under a closed embedding into $\mathbb{P}^{N'}$ that is associated to that line bundle. Let $W^{(i)}Ê\subset \Gamma(X_i,\mathcal{L}_i)$ be a fixed system of homogeneous coordinates on $X_i$. We identify $p_i^{\ast}W^{(i)}$ with $W^{(i)}$.

We fix a non-empty open subset $U^0$ of $X$ and suppose that there exists a very ample line bundle $\mathcal{P}$ on $X$, an injection $\mathcal{P} \hookrightarrow \mathcal{N}_a^{\otimes t_1}$ which induces an isomorphism on $U^0$ and a system of homogeneous coordinates $\Xi$ for $\mathcal{P}$ which are (by means of the aforementioned injection) monomials of multidegree $t_1a$ in the $W^{(i)}$.

We also suppose that there exists an injection $(\mathcal{P} \otimes \mathcal{M}^{\otimes -1}) \hookrightarrow \mathcal{N}_a^{\otimes t_2}$ which induces an isomorphism on $U^0$ and that $\mathcal{P} \otimes \mathcal{M}^{\otimes -1}$ is generated by a family $Z$ of $M$ global sections on $X$ which are polynomials of multidegree $t_2a$ in the $W^{(i)}$ such that the height of the family of coefficients of all these polynomials is at most $\sum_{i} a_i \delta_i$. 

The parameters $t_1, t_2, M \in \mathbb{N}$ and $\delta_1, \hdots, \delta_m \in [1,\infty)$ are fixed independently of the pair $(a,\mathcal{M})$. Using this pair, we can define the following two notions of height for a point $xÊ\in U^0(\bar{\mathbb{Q}})$:
\[Êh_{\mathcal{M}}(x) = h(\Xi(x)) - h(Z(x)),\]
\[Êh_{\mathcal{N}_a}(x) = a_1h\left(W^{(1)}(x)\right) + \cdots + a_mh\left(W^{(m)}(x)\right).\]

Let $\thetaÊ\geq 1$ and $\omega \geq -1$ be two integer parameters and set (with $\omega' = 3 +Ê\omega$)
\[Ê\Lambda = \theta(2t_1u_0)^{u_0}\left(\max_{1Ê\leq i \leq m}{N_i}+1\right)\prod_{i=1}^{m}{\deg(X_i)},\]
\[Ê\psi(u) = \prod_{j=u+1}^{u_0}{(\omega'j+1)}, \]
\[Êc_1 = c_2 = \Lambda^{\psi(0)},\]
\[Êc^{(i)}_3 = \Lambda^{2\psi(0)}(Mt_2)^{u_0}(h(X_i)+\delta_i)Ê\quad (i=1,\hdots,m),Ê\]
where $u_0 = \dim(X_1)+\cdots+\dim(X_m)$, $N_i+1 = \#W^{(i)}$ and the degrees and heights are computed with respect to the embeddings given by the $W^{(i)}$.

\begin{thm}\label{thm:vojta}
Let $x \in U^0(\bar{\mathbb{Q}})$ and let $(a,\mathcal{M})$ be a pair as defined above. Suppose that, for every subproduct of the form $Y = Y_1Ê\times \cdots \times Y_m$, where $Y_i \subset X_i$ is an irreducible subvariety that contains $x_i$, the following estimate holds:
\[Ê\mathcal{M}^{\cdot \dim(Y)}Ê\cdotÊY \geq \theta^{-1}\prod_{i=1}^{m}{(\deg(Y_i))^{-\omega}a_i^{\dim(Y_i)}}.\]
If furthermore $c_2a_{i+1} \leq a_i$ for every $i < m$ and $c^{(i)}_3 \leq h\left(W^{(i)}(x_i)\right)$ for every $iÊ\leq m$, then we have
\[Êh_{\mathcal{N}_a}(x) \leq c_1 h_{\mathcal{M}}(x).\]
\end{thm}

\begin{proof}
See \cite{D182} with $\mathcal{X} = X$ and $\pi = \id_X$.
\end{proof}

\bibliographystyle{acm}
\bibliography{Bibliography}

\end{document}